\documentclass[12pt]{amsart}
\makeatletter

\theoremstyle{plain}

\newtheorem{thm}{Theorem}[section]

\newtheorem{cor}[thm]{Corollary}

\newtheorem{lem}[thm]{Lemma}
\newtheorem{prop}[thm]{Proposition}
\theoremstyle{definition}
\newtheorem{defn}[thm]{Definition}
\theoremstyle{remark}
\newtheorem{rem}[thm]{Remark}
\theoremstyle{remark}

\theoremstyle{remark}

\theoremstyle{remark}

\theoremstyle{definition}
\newtheorem{example}[thm]{Example}

\theoremstyle{definition}

\theoremstyle{plain}

\theoremstyle{definition}

\theoremstyle{remark}

\theoremstyle{remark}

\theoremstyle{definition}

\theoremstyle{remark}
\newtheorem*{acknowledgement*}{Acknowledgement}

\newcommand{\R}{{\mathbb R}}

\newcommand{\C}{{\mathbb C}}
\newcommand{\N}{{\mathbb N}}
\newcommand{\Z}{{\mathbb Z}}

\newcommand{\e}{\varepsilon}

\newcommand{\hm}{{\mathcal M}}
\newcommand{\hn}{{\mathcal N}}
\newcommand{\hr}{{\mathcal R}}
\newcommand{\id}{\mathrm{id}}
\DeclareMathOperator{\tr}{tr}

\newcommand{\ip}[1]{\langle#1\rangle}

\DeclareMathOperator{\Ad}{Ad}

\DeclareMathOperator{\Aut}{Aut}
\newcommand{\U}{\mathcal{U}}

\newcommand{\T}{\mathbb{T}}

\DeclareMathOperator{\diag}{diag}

\newcommand{\rdim}{\mathrm{dim}_{\mathrm{Rok}}}

\def\freeprod{\font\bigsymbolsfont=cmsy10 scaled \magstep3
 \setbox0=\hbox{\bigsymbolsfont\char'003 }\mathord{\lower1pt\box0}}\relax\ignorespaces

\newcommand{\Hawaii}{Hawai\kern.05em`\kern.05em\relax i}

\usepackage{amssymb} 
\usepackage{amscd} 
\usepackage{hyperref}  
\usepackage{tikz} 
\usetikzlibrary{matrix,arrows}

\makeatother

\begin{document}

\title{A Rokhlin type theorem for simple C$^*$-algebras of finite nuclear dimension}

\author{Hung-Chang Liao}

\address{Department of Mathematics, Penn State University, State
	College, PA 16802, USA}

\email{hxl255@psu.edu} 

\begin{abstract}
We study $\Z$-actions on unital simple separable stably finite $C^*$-algebras of finite nuclear dimension. Assuming that the extreme boundary of the trace space is compact and finite dimensional, and that the induced action on the trace space is trivial, we show that strongly outer $\Z$-actions have finite Rokhlin dimension in the sense of Hirshberg, Winter and Zacharias. 
\end{abstract}

\maketitle


\section{Introduction} 

An important classical result in ergodic theory is Rokhlin's Lemma, which asserts that an ergodic measure preserving transformation on a non-atomic measure space can be approximated by cyclic shift in a suitable sense. In the 1970s Alain Connes made several fundamental breakthroughs in the study of automorphisms of von Neumann algebras (cf. \cite{Con01}). A key ingredient which led to the complete classification of $\mathbb{Z}$-actions on the hyperfinite II$_1$-factor up to outer conjugacy was his discovery that suitably outer automorphisms always have the so-called Rokhlin property, a noncommutative analogue of Rokhlin's Lemma. Up to approximation, the Rokhlin property requires the existence of a partition of unity (called a Rokhlin tower) consisting of mutually orthogonal projections such that the projections are cyclicly permuted by the automorphism. The notion of Rokhlin property was later imported to the realm of $C^*$-algebras. In the 1990s Kishimoto proved similar $C^*$-results for certain AF and A$\T$ algebras  (cf. \cite{Ki01,Ki03}). Later Nakamura proved an analogue for $\Z^2$-actions on UHF algebras (cf. \cite{Na01}), and $\Z$-actions on Kirchberg algebras \cite{Na02}.  However, the original definition of Rokhlin property turns out to be quite restrictive for $C^*$-algebras, especially the stably finite ones. There have been many attempts of generalizing the definition (e.g. \cite{OsPh01, Sa01}) and obtaining the corresponding Rokhlin type theorems (e.g. \cite{OsPh02, MaSa02}). Recently, Hirshberg, Winter and Zacharias introduced the notion of Rokhlin \textit{dimension} (cf. \cite{HWZ01}), a ``colored'' version of Rokhlin property motivated by the idea of covering dimension. It is proved in \cite{HWZ01} that this generalization preserves important regularity properties at the level of crossed products, and applies to a much larger class of $C^*$-algebra.  With the notion of Rohklin dimension in hand, one can hope to prove very general analogues of Connes' result. Indeed, very recently Matui and Sato made an important breakthrough, showing that for simple stably finite $C^*$-algebras with finite nuclear dimension and finitely many extremal traces, a $\Z$-action has finite Rokhlin dimension if and only if it is strongly outer (cf. \cite[Introduction]{HWZ01}). Inspired by their result and recent development of extending the Toms-Winter conjecture to a $C^*$-algebra whose trace space is a Bauer simplex, we provide the following result:

\begin{thm}
\label{thm:main}
Let $A$ be a unital separable simple stably finite nuclear $C^*$-algebra with property (SI). Suppose the extreme boundary $\partial T(A)$ of the trace space $T(A)$ is closed and has finite topological covering dimension. Let $\alpha: \Z\to \Aut(A)$ be a strongly outer $\Z$-action satisfying $T(A) = T(A)^\alpha$, where $T(A)^\alpha = \{\tau\in T(A): \tau\circ \alpha = \tau \}$. Then the Rokhlin dimension of $\alpha$ is no greater than one.
\end{thm}

Property (SI), as introduced by Sato (\cite{Sa01}), is a technical condition that allows one to compare tracially small elements with tracially large elements. The precise definition will be given in Section 6. Here we simply remark that property (SI) is possessed by a huge class of $C^*$-algebras. For example, a highly non-trivial result of Matui and Sato says that every unital simple separable nuclear stably finite $\mathcal{Z}$-stable $C^*$-algebra has property (SI) (\cite{MaSa01}). Note that R\o rdam's result that $\mathcal{Z}$-stability implies strict comparison (\cite{Ro01}) plays an important role there. In view of Winter's remarkable theorem connecting finite nuclear dimension and $\mathcal{Z}$-stability (\cite{Wi01}), every unital simple separable stably finite $C^*$-algebra of finite nuclear dimension has property (SI).

Every $\Z$-action $\alpha:\Z\to \Aut(A)$ naturally induces an action on the trace space $T(A)$. The assumption $T(A) = T(A)^\alpha$ amounts to saying that the induced action on $T(A)$ is trivial. This assumption is trivially satisfied when $A$ has a unique trace, or more generally, when $\alpha$ is approximately inner. We do not know at this point whether this assumption is essential. However, in the case that $\partial T(A)$ is a finite set, the argument can be modified to prove finite Rokhlin dimension regardless of the induced action on the trace space (as mentioned earlier this case was first obtained by Matui and Sato).

The converse of Theorem \ref{thm:main} holds without assumptions on the trace space and the induced action. We like to think of having finite Rokhlin dimension as a strong form of outerness. We will show in Section 2 that whenever $\alpha$ has finite Rokhlin dimension, the restriction map $r:T(A\rtimes_\alpha \Z)\to T(A)^\alpha$ must be injective. From this we deduce, using Kishimoto's result (\cite{Ki03}), that the action $\alpha$ is strongly outer. The following corollary summarizes these observations.

\begin{cor}\label{cor:main}
Let $A$ be a unital separable simple stably finite $C^*$-algebra of finite nuclear dimension. Suppose the extreme boundary $\partial T(A)$ of the trace space $T(A)$ is closed and has finite topological covering dimension. Let $\alpha: \Z\to \Aut(A)$ be a $\Z$-action. Assume one of the following holds:
\begin{enumerate}
\item[(a)] $T(A) = T(A)^\alpha$;
\item[(b)] $\partial T(A)$ is a finite set.
\end{enumerate}
Then the following are equivalent.
\begin{enumerate}
\item[(1)] The Rokhlin dimension of $\alpha$ is no greater than one.
\item[(2)] For every $m\in \N$, the restriction map $r:T(A\rtimes_{\alpha^m} \Z)\to T(A)^{\alpha^m}$ is injective.
\item[(3)] The action $\alpha:\Z\to \Aut(A)$ is strongly outer.
\end{enumerate}
\end{cor}

\noindent\textbf{Acknowledgement:}The author would like to thank Hiroki Matui and Yasuhiko Sato for sharing their unpublished work. Their result and proof will be described and reproduced in Section 6. The author also likes to thank his doctoral dissertation advisor Nate Brown, for many valuable discussions.


\section{Preliminaries and Overview}

In this section we present the most important definitions we need throughout the paper. We will provide a converse of Theorem \ref{thm:main}, as well as an overview of the proof of the main theorem.

Let us first recall the definition of Rokhlin dimension for $\Z$-actions.

\begin{defn} (\cite[Definition 2.3]{HWZ01} ) Let $A$ be a unital $C^*$-algebra and $\alpha:\Z\to \Aut(A)$ an action of the integers. We say $\alpha$ has \textit{Rokhlin dimension $d$},
$$
\dim_{Rok}(\alpha) = d,
$$
if $d$ is the least natural number such that the following holds: for any finite subset $F\subseteq A$, any $p\in \N$ and any $\e > 0$, there are positive contractions
$$
f_{0,0}^{(\ell)},...,f_{0,p-1}^{(\ell)}, f_{1,0}^{(\ell)},...,f_{1,p}^{(\ell)},\;\;\;\;\;\; \ell\in \{0,...,d\},
$$
in $A$ satisfying
\begin{enumerate}
\item[(1)] for any $\ell\in \{0,...,d\}, \| f_{q,i}^{(\ell)}f_{r,j}^{(\ell)} \| < \e$, whenever $(q,i) \neq (r,j)$,
\item[(2)] $\left\|  \sum_{\ell=0}^d \sum_{r=0}^1 \sum_{j=0}^{p-1+r} f_{r,j}^{(\ell)} - 1 \right\| < \e$,
\item[(3)] $\left\| \alpha_1(f_{r,j}^{(\ell)}) - f_{r,j+1}^{(\ell)}  \right\| < \e$ for all $r\in \{0,1\}$, $j\in \{0,1,...,p-2+r\}$ and $\ell\in \{0,...,d\}$,
\item[(4)] $\left\| \alpha_1( f_{0,p-1}^{(\ell)} + f_{1,p}^{(\ell)}   ) - ( f_{0,0}^{(\ell)} + f_{1,0}^{(\ell)}  )  \right\| < \e$ for all $\ell$,
\item[(5)] $\left\|  [f_{r,j}^{(\ell)}, a  ]  \right\| < \e$ for all $r,j,\ell$ and $a\in F$.
\end{enumerate}
If one can obtain this property such that for all $\ell$ one of the towers $(f_{0,j}^{(\ell)})$ or $(f_{1,j}^{(\ell)})$ vanishes, then we say that $\alpha$ has \textit{Rokhlin dimension $d$ with single towers}.
\end{defn}

\begin{rem} \label{rem:towerheight}
By \cite[Remark 2.4(vi)]{HWZ01}, in the above definition, instead of requiring the towers to have heights $p$ and $p-1$, one can require that there are such towers for arbitrarily large $p$, but not necessarily for every $p$.
\end{rem}

We can reformulate the definition using the notion of central sequence algebra. Let $A$ be a unital separable $C^*$-algebra and $\omega$ a free ultrafilter on $\N$. Set
$$
c_\omega(A) := \{ (a_n)_n\in \ell^\infty(\N,A) :\lim_{n\to\omega}\|a_n\| = 0 \}
$$
and define the \textit{ultraproduct} of $A$ by
$$
A_\omega := \ell^\infty(\N, A)/c_\omega(A).
$$
Since $A$ embeds into $A_\omega$ as constant sequences, we can form the relative commutant $F(A):= A_\omega\cap A'$, called the \textit{central sequence algebra} of $A$. Note that if $\alpha:\Z\to \Aut(A)$ is a $\Z$-action, then $\alpha$ can be promoted to well-defined actions $\alpha_\omega:\Z\to \Aut(A_\omega)$ and $\alpha_\omega:\Z\to \Aut(F(A))$ by setting $\alpha_\omega((a_n)_n) = ( \alpha(a_n) )_n$. One of the most important features of the ultraproduct (or the central sequence algebra) is that approximate relations in $A$ become exact in $A_\omega$ or $F(A)$. Therefore, one can define the Rokhlin dimension by requiring the existence of positive contractions $f_{0,0}^{(\ell)},...,f_{0,p-1}^{(\ell)}, f_{1,0}^{(\ell)},...,f_{1,p}^{(\ell)}$ in $F(A)$ instead of $A$, and asking (1)-(4) to hold exactly (with $\alpha_1$ replaced by $\alpha_\omega$).

The importance of having finite Rokhlin dimension can be seen from the fact that forming a crossed product by such an automorphism preserves finiteness of nuclear dimension. A more precise statement can be found in \cite{HWZ01}. Here we shall focus on the outerness property. The proof of the following proposition is essentially contained in \cite{Ki02}.

\begin{prop} (cf. \cite[Lemma 4.3]{Ki02})
Let $A$ be a unital simple separable $C^*$-algebra such that $T(A)\neq \emptyset$, and let $\alpha:\Z\to \Aut(A)$ be a $\Z$-action. If $\alpha$ has finite Rokhlin dimension, then the restriction map $r:T(A\rtimes_\alpha \Z)\to T(A)^\alpha$ is injective.
\end{prop}

\begin{proof} 
Because of \cite[Proposition 2.8]{HWZ01}, without losing generality we may assume that $\alpha$ has Rokhlin dimension $d$ with single towers. Let $\varPhi:A\rtimes_\alpha \Z\to A$ be the canonical conditional expectation. Given $\varphi\in T(A\rtimes_\alpha \Z)$, we need to show that $\varphi = \varphi|_A \circ \varPhi$.  Let $u$ be the unitary in $A\rtimes_\alpha \Z$ implementing the action. It suffices to show that $\varphi(au^n) = 0$ for all $a\in A$ and $n\in \N$. 

Fix $n\in \N$, $a\in A$ with $\|a\| \leq 1$, and $\e > 0$. Put $p = 2n$ and $\varepsilon' = \frac{\e}{2(d+1)n+1}$. By the definition of the Rokhlin dimension (with single towers), we can find positive contractions $\{ f_{k}^{(\ell)}:1\leq k\leq p, \;0\leq \ell \leq d  \}$ in $A$ such that 
\begin{enumerate}
\item[(i)] $\left\| \sum_{\ell=0}^d \sum_{k=1}^p f_k^{(\ell)} - 1 \right\| < \e'$;
\item[(ii)] $\|  { f_k^{(\ell)} }^{ \frac{1}{2} }  a \alpha^n(  { f_k^{(\ell)} }^{\frac{1}{2}}  )  \| < \e'\;\;\;\; (\ell=0,1,...,d;\; k=1,...,p)  $
\end{enumerate}
(to see this, note that these two expressions are equal to zero in $F(A)$). Then
\begin{align*}
| \varphi(au^n) | &\leq \left|  \sum_{\ell=0}^d\sum_{k=1}^p \varphi(au^nf_k^{(\ell)})  \right| + \e' \\
&\leq \sum_{\ell=0}^d\sum_{k=1}^p \left|  \varphi(  {f_k^{(\ell)}}^{\frac{1}{2}}  a u^n  {f_k^{(\ell)}}^{\frac{1}{2}}  )   \right| + \epsilon' \\
&= \sum_{\ell=0}^d\sum_{k=1}^p \left|  \varphi({f_k^{(\ell)}}^{\frac{1}{2}} a \alpha^n( {f_k^{(\ell)}}^{\frac{1}{2}}   ) u^n   )    \right| + \e' \\
&\leq ((d+1)p + 1)\e' \\
&= (2(d+1)n+1 )\e' = \e.
\end{align*}
\end{proof}

Note that the implication $(1)\Longrightarrow (2)$ in Corollary \ref{cor:main} follows directly from the proposition above since if $\alpha$ has finite Rokhlin dimension, then so does each power of $\alpha$. Next we recall the definition of a strongly outer action from \cite{MaSa02}.

\begin{defn} (\cite[Definition 2.7]{MaSa02}) Let $A$ be a unital $C^*$-algebra with $T(A)\neq \emptyset$. An action $\alpha:G\to \Aut(A)$ of a discrete group $G$ on $A$ is \textit{strongly outer} if for every $g\in G\setminus\{e\}$ and $\tau\in T(A)^{\alpha_g}$, the weak extension $\tilde{\alpha}_g\in\Aut(\pi_\tau(A)'')$ is an outer automorphism.
\end{defn}

Now the implication $(2)\Longrightarrow (3)$ in Corollary \ref{cor:main} is obtained from the following result of Kishimoto:

\begin{prop}\label{prop:kishimotoouterness} (\cite[Proposition 2.3]{Ki03}) Let $A$ be a unital $C^*$-algebra with $T(A)\neq \emptyset$ and $\alpha$ be an automorphism of $A$. If $T(A\rtimes_\alpha \Z)$ is isomorphic with $T(A)^\alpha$ under the restriction map, then for every extreme point $\tau$ of $T(A)^\alpha$, the $\Z$-action on $\pi_\tau(A)''$ induced by $\alpha$ is outer.
\end{prop}

The proof of Theorem \ref{thm:main} (and hence the implication $(3)\Longrightarrow (1)$ in Corollary \ref{cor:main}) relies on the notion of continuous $W^*$-bundle introduced by Ozawa in \cite{Oz01}. We recall the basic definitions from \cite{Oz01} and \cite{BBSTWW}.

\begin{defn} (\cite[Section 5]{Oz01})
A \textit{continuous $W^*$-bundle over a compact metrizable space} (or simply a \textit{continuous $W^*$-bundle}) is a triple $(\hm,K,E)$, where
\begin{itemize}
\item $\hm$ is a $C^*$-algebra,
\item $K$ is a compact metrizable space such that $C(K)\subseteq Z(\hm)$, the center of $\hm$,
\item $E:\hm\to C(K)$ is a faithful conditional expectation with $E(ab) = E(ba)$ for all $a,b\in \hm$
\end{itemize}
such that the norm-closed unit ball of $\hm$ is complete in the \textit{uniform 2-norm} $\|\cdot\|_{2,u}$ defined by $\|x\|_{2,u} := \| E(x^*x)\|^{1/2}$. We say $(\hm, K, E)$ is \textit{separable} if $\hm$ contains a countable $\|\cdot\|_{2,u}$-dense subset. For each $\lambda\in K$, the GNS representation $\pi_\lambda(\hm)$ coming from the tracial state $ev_\lambda\circ E$ is called a \textit{fiber} of $(\hm, K, E)$.
\end{defn}

We shall say a few words about the ``continuity'' of a continuous $W^*$-bundle $(\hm, K, E)$. For each $x\in \hm$, observe that $E(x^*x)$ is a (positive) continuous function on $K$. Therefore for each $\lambda\in K$ and $\e > 0$ we can find a neighborhood $U$ of $\lambda$ such that
$$
\sup_{\mu\in U}\left| E(x^*x)(\mu)^{1/2} - E(x^*x)(\lambda)^{1/2}  \right| < \e.
$$
Write $\|x\|_{2,\lambda} := \tau_\lambda(x^*x)^{1/2} = E(x^*x)(\lambda)^{1/2}$. Then we arrive at the following ``2-norm continuity'':

\begin{prop}\label{prop:2normcontinuity} Let $(\hm, K, E)$ be a continuous $W^*$-bundle. Let $x\in \hm$ and $\lambda\in K$. Then for every $\e > 0$, there exists a neighborhood $U$ of $\lambda$ such that
$$
\sup_{\mu\in U}\left| \|x\|_{2,\mu} - \|x\|_{2,\lambda}   \right| < \e.
$$
\end{prop}

Next we discuss morphisms in the category of continuous $W^*$-bundles. Let $(\hm,K,E_\hm)$ and $(\hn,L,E_\hn)$ be continuous $W^*$-bundles. A \textit{morphism} from $(\hm,K,E_\hm)$ to $(\hn,L,E_\hn)$ is a pair $(\theta, \sigma)$, where
\begin{itemize}
\item $\theta:\hm\to\hn$ is a *-homomorphism,
\item $\sigma:L\to K$ is a continuous map, inducing $\bar{\sigma}:C(K)\to C(L)$,
\end{itemize}
such that the diagram
\begin{center}
\begin{tikzpicture}[node distance=2cm, auto]
  \node (1) {$\hm$};
  \node (2) [right of=1] {$\hn$};
  \node (6) [below of=1] {$C(K)$};
  \node (7) [below of=2] {$C(L)$};
  \draw[->] (1) to node {$\theta$} (2);
  \draw[->] (6) to node {$\bar{\sigma}$} (7);
  \draw[->] (2) to node {$E_\hn$} (7);
  \draw[->] [swap] (1) to node {$E_\hm$} (6);
\end{tikzpicture}.
\end{center}
commutes (cf. \cite{BBSTWW}). When $\theta$ is a *-isomorphism and $\sigma$ is a homeomorphism, we say that $(\theta, \sigma)$ is an \textit{isomorphism} of continuous $W^*$-bundles. In the case that $(\hn, E_\hn, L) = (\hm, E_\hm, K)$, we say $(\theta, \sigma)$ is an \textit{automorphism} of $(\hm, K, E)$.

One of the motivations of studying continuous $W^*$-bundles is the following example:

\begin{example} \label{exam:2-norm completion} (Uniform 2-norm completion; cf. \cite{Oz01, BBSTWW}) Let $A$ be a unital separable $C^*$-algebra with $T(A)\neq \emptyset$. Define the \textit{uniform 2-norm} $\|\cdot\|_{2,u}$ on $A$ by
$$
\|a\|_{2,u} := \sup_{\tau\in T(A)} \tau(a^*a)^{1/2}.
$$
Let $\overline{A}^u$ be the $C^*$-algebra of norm-bounded uniform 2-norm Cauchy sequences, modulo the ideal of uniform 2-norm null sequences, i.e., $\overline{A}^u$ is the \textit{uniform 2-norm completion} of $A$.

Ozawa proved in \cite{Oz01} that when $\partial T(A)$ is compact, there exists a (faithful tracial) conditional expectation $E:\overline{A}^u\to C(\partial T(A))$ satisfying $E(a)(\tau) = \tau(a)$ for each $a\in A$ and $\tau\in \partial T(A)$. This gives $(\overline{A}^u, \partial T(A), E)$ the structure of a continuous $W^*$-bundle. Ozawa also showed that the fibers are given by $\pi_\tau(\overline{A}^u) \cong \pi_\tau(A)''$ (\cite[Theorem 3]{Oz01}). In the case that $A$ is simple and nuclear, all fibers are isomorphic to the hyperfinite II$_1$ factor $\hr$.
\end{example}

Let us describe the strategy of proving Theorem \ref{thm:main}. For this discussion, we write $(\hm, K, E)$ for the continuous $W^*$-bundle $(\overline{A}^u, \partial T(A), E )$, as defined in Example \ref{exam:2-norm completion}. We apply Connes' non-commutative Rokhlin theorem \cite{Con01} to the outer action induced by $\alpha$ on each fiber $\pi_\tau(A)''$. This produces a Rokhlin tower consisting of projections in each fiber of the bundle $(\hm, K, E)$. Now the first objective is to glue these towers together and build a global Rokhlin tower in the bundle $(\hm, K, E)$. A standard arugument for this kind of ''local to global'' result is to make use of a partition of unity. However, the naive approach, namely taking a partition of unity from $C(K)$, would not work, since any overlap of the open subsets of $K$ would destroy the orthogonality among the Rokhlin elements. 

Therefore we take a more indirect route, which consists of two steps. The first step is to establish the fact that the fixed point subalgebra of the central sequence $W^*$-bundle $\hm^{\omega}\cap\hm'$ (see Section 3 for the precise definition) contains a unital copy of the matrix algebra $M_p$ for each $p\in \N$ (see Proposition \ref{prop:fixpointalgebra}). This is done by gluing a fiberwise result due to Ocneanu (\cite{Oc01}) using an ordinary partition of unity. We remark that the topological finite dimensionality of the extreme boundary plays a crucial role here. In fact, by an observation of Kirchberg and R\o rdam (Proposition \ref{prop:fixpointorderzero}), it suffices to produce finitely many order zero maps, i.e., completely positive maps preserving orthogonality (see the paragraph before Proposition \ref{prop:fixpointorderzero}), with commuting ranges from the matrix algebra into the fixed point subalgebra of $\hm^\omega\cap\hm'$. The finiteness of topological dimensionality comes into play here and is closely related to the number of order zero maps.

The second step is to utilize the extra room provided by these matrix algebras and obtain a ''mutually orthogonal'' partition of unity in the bundle $(\hm, K, E)$. The idea, largely due to Ozawa (cf. \cite{Oz01}), is that given an ordinary partition of unity $g_1,g_2,...,g_N$ in $C(K)\subseteq Z(\hm)$, for each $\lambda\in K$ we can divide the identity in $M_p$ into positive elements $p_1(\lambda),p_2(\lambda),...,p_N(\lambda)$ so that $\sum_{k=1}^N p_k(\lambda) = 1_p$ and the (normalized) trace of $p_k(\lambda)$ coincides with the value of $g_k$ at $\lambda$ (see the proof of Lemma \ref{lem:bundlepartition}). Since the uniform 2-norm on $\hm$ essentially arises from traces, (the global sections obtained from) these elements in some sense are good ''simulations'' of the functions $g_1,g_2,...,g_N$. The crucial point of this construction is that when $p$ is sufficiently large, the positive elements $p_1,p_2,...,p_N$ become almost mutually orthogonal projections in terms of the uniform 2-norm. Now we are in a position to establish a ''global'' Rokhlin type theorem for the bundle $(\hm, K, E)$ using this new partition of unity (see Theorem \ref{thm:bundleRokhlin}).

One of the key ingredients appearing in the interaction of von Neumann algebras and $C^*$-algebras is the surjectivity of the canonical map $F(A)\to N^\omega\cap N'$ between central sequence algebras, proven by Sato \cite{Sa02} under a nuclearity assumption and by Kirchberg and R\o rdam \cite{KirRo01} in general. Here $N$ is the weak closure of $A$ in a tracial representation. In Section 5 we apply this technique, together with an invariant approximate unit trick (see Proposition \ref{prop:alphasigmaideal}), to lift the Rokhlin tower from $\hm^\omega\cap \hm'$ back to $F(A)$. This gives us the so-called \textit{approximate Rokhlin property} for $A$. The difference between the approximate Rokhlin property and the genuine Rokhlin property is that in the approximate version the sum of the Rokhlin elements approximates $1_A$ only up to the uniform 2-norm. To correct this defect, in Section 6 we reproduce a matrix algebra technique of Matui and Sato (the technique is a variation of what appeared first in Kishimoto's work on UHF algebras; cf. \cite{Ki02}) so that we can merge the defect element to one of the Rokhlin towers, and therefore complete the proof.


\section{Fixed point subalgebra of central sequence $W^*$-bundles}

We discuss briefly the construction of the ultraproduct of a sequence of continuous $W^*$-bundles $(\hm_n,K_n,E_n)$ with uniform 2-norms $\|\cdot\|_{2,u}^{(n)}$. This part is largely taken from \cite{BBSTWW}. Let $\prod^\omega \hm_n$ be the $C^*$-algebra consisting of norm-bounded sequence $(x_n)_{n=1}^\infty$, where each $x_n$ belongs to $\hm_n$,  modulo the ideal of uniform 2-norm null sequences. Define the uniform 2-norm on $\prod^\omega \hm_n$ by
$$
\|(x_n)_n\|_{2,u}^{(\omega)} := \lim_{n\to\omega} \|x_n\|_{2,u}^{(n)}.
$$
Let $\prod_\omega C(K_n)$ be the $C^*$-ultraproduct of the sequence of $C^*$-algebras $(C(K_n))_n$, and let $\Sigma_\omega K_n$ denote the \textit{ultracoproduct} of $(K_n)_n$, defined to be the space satisfying $C(\Sigma_\omega K_n) \cong \prod_\omega C(K_n)$ via the Gelfand duality.

\begin{prop}\label{prop:bundleultraproduct} $($\cite[Proposition 3.9]{BBSTWW}$)$ Let $(\hm_n, K_n, E_n)$ be a sequence of continuous $W^*$-bundles and let $\omega$ be a free ultrafilter on $\N$. Then the ultraproduct $\prod^\omega \hm_n$ is a continuous $W^*$-bundle over $\Sigma_\omega K_n$ via the conditional expectation $E^\omega:\prod^\omega \hm_n\to \prod_\omega C(K_n)$ induced by $E_n:\hm_n\to C(K_n)$. Moreover,
$$
\| E^\omega(x^*x)\|^{1/2} = \|x\|_{2,u}^{(\omega)}\;\;\;\;\;\; (x\in \prod^\omega \hm_n).
$$
\end{prop}

In the case that $\hm_n = \hm$ and $K_n = K$ for each $n\in \N$, we write $(\hm^\omega, K^\omega, E^\omega)$ or simply $\hm^\omega$ for the ultraproduct. As usual, the canonical embedding $\hm\to \hm^\omega$ allows one to form the \textit{central sequence $W^*$-bundle} $\hm^\omega\cap \hm'$.

Next we consider automorphisms of continuous $W^*$-bundles. Suppose $(\beta,\sigma)$ is an automorphism of a continuous $W^*$-bundle $(\hm, K, E)$. Then for each $\lambda\in K$ there is a well-defined *-isomorphism, called the \textit{fiber map}, $\beta_\lambda:\pi_\lambda(\hm)\to \pi_{\sigma^{-1}(\lambda)}(\hm)$ defined by
$$
\beta_\lambda(\pi_\lambda(x)) = \pi_{\sigma^{-1}(\lambda)}(\beta(x))\;\;\;\;\;\;\;\; (x\in\hm).
$$
In the case that $\sigma$ is the identity map on $K$, the map $\beta_\lambda$ gives rise to an automorphism of $\pi_\lambda(\hm)$.

Every automorphism $(\beta,\sigma)$ of $(\hm,K,E)$ can be promoted to an automorphism of the ultraproduct $(\hm^\omega, K^\omega, E^\omega)$ or the central sequence $W^*$-bundle $(\hm^\omega\cap \hm',K^\omega, E^\omega)$ in the following way. Define
$$
\beta_\omega:\hm^\omega\to\hm^\omega,\;\;\; \beta_\omega((x_n)_n) = (\beta(x_n))_n
$$
and 
$$
\bar{\sigma}_\omega:C(K^\omega) \to C(K^\omega),\;\;\; \bar{\sigma}_\omega((f_n)_n) = (f_n\circ \sigma)_n
$$
(here we identify $C(K^\omega)$ with $\prod_\omega C(K)$). It is clear that $\beta_\omega$ and $\bar{\sigma}_\omega$ are automorphisms of $C^*$-algebras. Let $\sigma_\omega:K^\omega\to K^\omega$ be the map induced by $\bar{\sigma}_\omega$ via the Gelfand duality. In order to show that $(\beta_\omega, \sigma_\omega)$ form a $W^*$-bundle morphism, we make a small digression. The set theoretic \textit{ultraproduct} $\prod_\omega K_n$ (or $K_\omega$ when $K_n = K$ for all $n$) of a sequence of compact spaces $K_n$ is defined to be the product space $\prod_{n=1}^\infty K_n$ modulo the equivalence relation given by $(\lambda_n)\sim (\mu_n)$ if and only if the set $\{ n:\lambda_n = \mu_n \}$ belongs to the ultrafilter $\omega$ (cf. \cite{BBSTWW}). The ultraproduct $\prod_\omega K_n$ can be identified with a dense subset of the ultracoproduct $\sum_\omega K_n$ via the map
$$
\prod_\omega K_n\to \left( \prod_\omega C(K_n)  \right)^*,\;\;\; (\lambda_1,\lambda_2,...)\mapsto \lambda,
$$
where $\lambda:\prod_\omega C(K_n)\to \C$ is the character defined by $\lambda((g_n)_n) = \lim_{n\to \omega} g_n(\lambda_n)$.

\begin{prop}\label{prop:automorphism}
Let $(\beta,\sigma)$ be an automorphism of a continuous $W^*$-bundle $(\hm, K, E)$. Then $(\beta_\omega, \sigma_\omega)$ is an automorphism of $(\hm^\omega, K^\omega, E^\omega)$.
\end{prop}
\begin{proof}
We need to show that the diagram
\begin{center}
\begin{tikzpicture}[node distance=2cm, auto]
  \node (1) {$\hm^\omega$};
  \node (2) [right of=1] {$\hm^\omega$};
  \node (6) [below of=1] {$C(K^\omega)$};
  \node (7) [below of=2] {$C(K^\omega)$};
  \draw[->] (1) to node {$\beta_\omega$} (2);
  \draw[->] (6) to node {$\bar{\sigma}_\omega$} (7);
  \draw[->] (2) to node {$E^\omega$} (7);
  \draw[->] [swap] (1) to node {$E^\omega$} (6);
\end{tikzpicture}.
\end{center}
is commutative. Let $x$ be an element in $\hm^\omega$ represented by $(x_1,x_2,...)$ and let $\lambda\in K_\omega\subseteq K^\omega$ be represented by $(\lambda_1,\lambda_2,...)$. We compute
\begin{align*}
\left[ (\bar{\sigma}_\omega\circ E^\omega)(x)  \right](\lambda) &= \lim_{n\to \omega}[E(x_n)](\sigma(\lambda_n)) \\
&= \lim_{n\to \omega}[(\bar{\sigma}\circ E )(x_n)](\lambda_n) \\
&= \lim_{n\to\omega} [(E\circ \beta)(x_n)](\lambda_n) \\
&= [(E^\omega\circ \beta_\omega)(x)](\lambda).
\end{align*}
Since $K_\omega$ is dense in $K^\omega$, we have shown that $\bar{\sigma}_\omega \circ E^\omega = E^\omega\circ \beta_\omega$.
\end{proof}

We shall write $(\hm^\omega\cap \hm')^{\beta_\omega}$ for the fixed point subalgebra, namely the set of elements $x$ in $\hm^\omega\cap \hm'$ satisfying $\beta_\omega(x) = x$. Here is the main result of this section.

\begin{prop}
\label{prop:fixpointalgebra}
Let $(\hm, K, E)$ be a separable continuous $W^*$-bundle such that $\dim(K) = m < \infty$ and $\pi_\lambda(\hm)\cong \hr$ for every $\lambda\in K$. Suppose $(\beta, \id)$ is an automorphism of $(\hm, K, E)$, and suppose the fiber map $\beta_\lambda\in \Aut(\pi_\lambda(\hm))$ is pointwise outer for each $\lambda\in K$. Then for every $p\in \N$, there exists a unital *-homomorphism $\psi:M_p \to (\hm^\omega\cap \hm')^{\beta_\omega}$. 
\end{prop}

Recall that a completely positive (c.p.) map $\varphi:A\to B$ between $C^*$-algebras is called \textit{order zero} if it preserves orthogonality, i.e., $\varphi(a)\varphi(b) = 0$ whenever $a$ and $b$ are positive elements satisfying $ab=0$. In view of \cite[Lemma 7.6]{KirRo01}, it suffices to show the following:

\begin{prop}
\label{prop:fixpointorderzero}
Under the same assumptions as Proposition \ref{prop:fixpointalgebra}, for every $p\in \N$, there exist completely positive contractive (c.p.c.) order zero maps $\psi_1,\psi_2,...,\psi_{m+1}:M_p\to (\hm^\omega\cap \hm')^{\beta_\omega}$ with commuting ranges such that $\psi_1(1)+ \cdots + \psi_{m+1}(1) = 1$.
\end{prop}

We shall construct the order zero maps $\psi_1,\psi_2,...,\psi_{m+1}$ in a way very similar to \cite{KirRo01}. Since the proof is quite lengthy, we divide it into several lemmas.

\begin{lem}
\label{lem:fixpointnbd}
Under the same assumptions as Proposition \ref{prop:fixpointalgebra}, for every $\lambda\in K, \e > 0, p\in \N$, and norm bounded $\|\cdot\|_{2,u}$-compact subset $\Omega\subseteq \hm$, there exists a neighborhood $O_\lambda$ of $\lambda$ and a unital completely positive (u.c.p.) map $\varphi:M_p \to \hm$ such that
\begin{enumerate}
\item[(i)] $\sup_{\mu\in O_\lambda}\| [\varphi(e), x]\|_{2,\mu} < \e$\;\;\;\;\;\; $(e\in (M_p)_1,\; x\in \Omega)$;
\item[(ii)] $\sup_{\mu\in O_\lambda}\| \varphi(e^*e) - \varphi(e)^*\varphi(e) \|_{2,\mu} < \e$\;\;\;\;\;\; $(e\in (M_p)_1)$;
\item[(iii)] $\sup_{\mu\in O_\lambda}\| \beta(\varphi(e)) - \varphi(e) \|_{2,\mu} < \e$\;\;\;\;\;\; $(e\in (M_p)_1)$.
\end{enumerate}
\end{lem}
\begin{proof}
By \cite{Oc01} (cf. \cite[Lemma 4.1]{MaSa03}), $\left( \pi_\lambda(\hm)^\omega\cap \pi_\lambda(\hm)'  \right)^{\beta_{\lambda,\omega}}$ contains a unital copy of $M_p$. Therefore there exists a sequence $(\varphi_n)_{n=1}^\infty$ of u.c.p. maps $\varphi_n:M_p\to \pi_\lambda(\hm)$ such that
\begin{enumerate}
\item[(a)] $\lim_{n\to\omega} \left\| [\varphi_n(e), \pi_\lambda(x)] \right\|_{2,\lambda}   = 0\;\;\;\;\;\;\;\; (e\in M_p; x\in \Omega)$;
\item[(b)] $\lim_{n\to\omega}  \left\| \varphi_n(e^*e) - \varphi(e)^*\varphi(e) \right\|_{2,\lambda}   = 0\;\;\;\;\;\;\;\; (e\in M_p)$;
\item[(c)] $\lim_{n\to\omega}  \left\| \beta_\lambda(\varphi_n(e)) - \varphi_n(e) \right\|_{2,\lambda}  = 0\;\;\;\;\;\;\;\ (e\in M_p)$.
\end{enumerate}

Since $(M_p)_1$ and $\Omega$ are compact, choosing $n$ large enough along the ultrafilter $\omega$, and lifting along the quotient map $\pi_\lambda:\hm\to \pi_\lambda(\hm)$, we obtain a u.c.p. map $\varphi:M_p\to \hm$ satisfying
\begin{enumerate}
\item[(a)'] $\sup_{e\in (M_p)_1, x\in \Omega} \| [\varphi(e), x] \|_{2,\lambda} < \e$;
\item[(b)'] $\sup_{e\in (M_p)_1} \| [\varphi(e^*e)- \varphi(e)^*\varphi(e) \|_{2,\lambda} < \e$;
\item[(c)'] $\sup_{e\in (M_p)_1} \| \beta(\varphi(e)) - \varphi(e) \|_{2,\lambda} < \e  $.
\end{enumerate}
By the 2-norm continuity (Proposition \ref{prop:2normcontinuity}) of continuous $W^*$-bundles, the map
$$
(M_p)_1\times \Omega\times K\to \R,\;\;\;\;\;\; (e,x,\lambda)\mapsto \| [\varphi(e),x]\|_{2,\lambda}
$$
is continuous. Since $(M_p)_1\times \Omega\times K$ is compact, the estimate (a)' holds for some small neighborhood $O_1$ of $\lambda$. Similarly we can find neighborhoods $O_2$ and $O_3$ of $\lambda$ such that the estimates (b)' and (c)' hold, respectively. Now the proof is finished by taking $O_\lambda := O_1\cap O_2\cap O_3$.
\end{proof}

The following lemma is essentially the same as \cite[Proposition 7.4]{KirRo01}.

\begin{lem}
\label{lem:fixpointcoversystem}
Under the same assumptions as Proposition \ref{prop:fixpointalgebra}, for every $\e> 0, p\in \N, m\in \N$, and norm bounded $\|\cdot\|_{2,u}$-compact subset $\Omega\subseteq \hm$, there exists a system\\ $\{ \mathcal{U}; \varPhi^{(1)},...,\varPhi^{(m)} \}$ consisting of
\begin{itemize}
\item an open covering $\mathcal{U}$ of $K$, and
\item finite collections $\varPhi^{(\ell)} = \{ \varphi_1^{(\ell)},...,\varphi_{h(\ell)}^{(\ell)} \}$ of u.c.p. maps $\varphi_j^{(\ell)}:M_p\to \hm$
\end{itemize}
satisfying the following condition: for each open set $O\in \mathcal{U}$ and each $\ell\in \{1,...,m\}$ there exists $j\in \{1,...,h(\ell)\}$ such that
\begin{enumerate}
\item[(i)] $\sup_{\mu\in O} \| [\varphi_j^{(\ell)}(e), x ] \|_{2,\mu} < \e$\;\;\;\;\;\; $(e\in (M_p)_1,\; x\in \hm)$;
\item[(ii)] $\sup_{\mu\in O} \| [ \varphi_j^{(\ell)}(e), \varphi_i^{(k)}(f) ]  \|_{2,\mu} <\e$\;\;\;\;\;\; $(1\leq k\neq \ell\leq m,\; 1\leq i\leq h(k),\; e,f\in (M_p)_1)$;
\item[(iii)] $\sup_{\mu\in O} \| \varphi_j^{(\ell)}(e^*e) - \varphi_j^{(\ell)}(e)^*\varphi_j^{(\ell)}(e)  \|_{2,\mu} < \e$\;\;\;\;\;\; $(e,f\in (M_p)_1)$;
\item[(iv)] $\sup_{\mu\in O} \|  \beta( \varphi_j^{(\ell)}(e) ) - \varphi_j^{(\ell)}(e)  \|_{2,\mu} < \e$\;\;\;\;\;\; $(e\in (M_p)_1)$.
\end{enumerate}
\end{lem}
\begin{proof}
We proceed by induction on $m$. First consider the case that $m=1$. By Lemma \ref{lem:fixpointnbd}, for each $\lambda\in K$, there exists a u.c.p. map $\varphi_\lambda:M_p\to \hm$ and a neighborhood $O_\lambda$ of $\lambda$ such that the condition (i)(iii)(iv) hold. Cover $K$ by $O_\lambda$ and extract a finite subcover $\mathcal{U} = \bigcup_{j=1}^{h(1)} O_{\lambda_j} $. Then $\{ \mathcal{U}; \varPhi^{(1)} = \{ \varphi_{\lambda_1},...,\varphi_{\lambda_{h(1)}}  \} \}$ form a required system.

Now suppose we have constructed the system $\{ \mathcal{U}'; \varPhi^{(1)},...,\varPhi^{(m-1)} \}$. By Lemma \ref{lem:fixpointnbd} for each $\lambda\in K$ there exists a u.c.p map $\varphi_\lambda^{(m)}:M_p\to \hm$ and a neighborhood $O_\lambda$ satisfying (i)(iii)(iv) with $\Omega$ replaced by $\Omega \cup \left(  \bigcup_{1\leq \ell\leq m-1,\; 1\leq j\leq h(\ell)}  \varphi_j^{(\ell)}( (M_p)_1 )  \right)$. Cover $K$ by $O_\lambda$ and extract a finite subcover $\mathcal{U}'' = \bigcup_{j=1}^{h(m)}O_{\lambda_j}$. Let $\mathcal{U}$ be the collection of all subsets of the form $O'\cap O''$, where $O'\in \mathcal{U}'$ and $O''\in \mathcal{U}''$, and let $\varPhi^{(m)} = \{ \varphi_1^{(m)},..., \varphi_{h(m)}^{(m)}  \}$. Then $\{  \mathcal{U}; \varPhi^{(1)},...,\varPhi^{(m)} \}$ form a required system.
\end{proof}

\begin{rem} The assumption on the topological covering dimension of $K$ is unnecessary in Lemma \ref{lem:fixpointnbd} and Lemma \ref{lem:fixpointcoversystem}.
\end{rem}

\begin{lem}
\label{lem:fixpointcpc}
Under the same assumptions as Proposition \ref{prop:fixpointalgebra}, for every $\e > 0, p\in \N$, and norm bounded $\|\cdot\|_{2,u}$-compact subset $\Omega\subseteq \hm$, there exist c.p.c. maps $\psi^{(1)},...,\psi^{(m+1)}:M_p\to \hm$ (recall that $m = \dim(K)$) such that
\begin{enumerate}
\item[(i)] $\|[\psi^{(\ell)}(e),x]\|_{2,u} < \e$\;\;\;\;\;\; $(1\leq \ell\leq m+1,\; e\in (M_p)_1,\; x\in \Omega)$;
\item[(ii)] $\| [\psi^{(\ell)}(e), \psi^{(k)}(f) ] \|_{2,u} < \e$\;\;\;\;\;\; $(1\leq \ell\neq k \leq m+1,\;e, f\in (M_p)_1)$;
\item[(iii)] $\| \psi^{(\ell)}(1)\psi^{(\ell)}(e^*e) - \psi^{(\ell)}(e)^*\psi^{(\ell)}(e) \|_{2,u} < \e$\;\;\;\;\;\; $(1\leq \ell\leq m+1,\; e\in (M_p)_1)$;
\item[(iv)] $\sum_{\ell=1}^{m+1}\psi^{(\ell)}(1)  = 1$;
\item[(v)] $\| \beta(\psi^{(\ell)}(e))-\psi^{(\ell)}(e)\|_{2,u} < \e$\;\;\;\;\;\; $(1\leq \ell\leq m+1,\; e\in (M_p)_1)$.
\end{enumerate}
\end{lem}
\begin{proof}
By Lemma \ref{lem:fixpointcoversystem}, we can find a system $(\mathcal{U}; \varPhi^{(1)},...,\varPhi^{(m+1)})$ satisfying the condition specified in the same lemma. Since $\dim(K) = m$, there is a $(m+1)$-decomposable open refinement $\mathcal{U}' = \mathcal{U}_1\cup \cdots \cup \mathcal{U}_{m+1}$ of $\mathcal{U}$, where each $\mathcal{U}_\ell$ is a finite collection of disjoint open sets $(O_1^{(\ell)},...,O_{s(\ell)}^{(\ell)})$. By construction, for each $\ell\in \{1,2,...,m+1\}$ and $j\in \{1,2,...,s(\ell)\}$ there exists a u.c.p. map $\varphi_j^{(\ell)}:M_p\to \hm$ from the collection $\varPhi^{(\ell)}$ such that (i)-(iv) in Lemma \ref{lem:fixpointcoversystem} hold.

Let $\{  g_j^{(\ell)}: \ell=1,2,...,m+1, j= 1,2,..., s(\ell)\}$ be a partition of unity subordinate to the open cover $\bigcup_{\ell,j} O_j^{(\ell)}$. Note that each $g_j^{(\ell)}$ is fixed by the automorphism $\beta$ of $\hm$ since we assume the action on $K$ is trivial. For each $\ell$ define the c.p.c map $\psi^{(\ell)}:M_p\to \hm$ by
$$
\psi^{(\ell)} = \sum_{j=1}^{s(\ell)} g_j^{(\ell)^{\frac{1}{2} }} \varphi_j^{(\ell)}(\cdot) g_j^{(\ell)^{\frac{1}{2} }}.
$$
We claim that $\{\psi^{(\ell)}\}_{1\leq \ell\leq m+1}$ satisfy conditions (i)-(v). Note that since each $\varphi_j^{(\ell)}$ is unital, $\varphi^{(\ell)}(1) = \sum_{j=1}^{s(\ell)} g_j^{(\ell)}$ and (iv) follows immediately. For the other estimates we shall only verify (i) since the arguments are similar. 

Given $e\in (M_p)_1$, $x\in \Omega$, and $\ell\in \{1,...,m+1\}$, we write
\begin{align*}
[\psi^{(\ell)}(e),x] &= \left[ \sum_{j=1}^{s(\ell)} g_j^{(\ell)^{\frac{1}{2}}} \varphi_j^{(\ell)}(e) g_j^{(\ell)^{\frac{1}{2} }}, x  \right] \\
&= \sum_{j=1}^{s(\ell)} g_j^{(\ell)^{\frac{1}{2} }} [ \varphi_j^{(\ell)}(e), x] g_j^{(\ell)^{\frac{1}{2} }}.
\end{align*}
Since $g_j^{(\ell)}g_i^{(\ell)} = 0$ whenever $i\neq j$, we have
$$
[\psi^{(\ell)}(e), x ]^*[\psi^{(\ell)}(e), x ] =  \sum_{j=1}^{s(\ell)} g_j^{(\ell)^{\frac{1}{2}}} [\varphi_j^{(\ell)}(e), x ]^*  g_j^{(\ell)}  [\varphi_j^{(\ell)}(e), x ]  g_j^{(\ell)^{\frac{1}{2}}}.
$$
By definition of the uniform 2-norm, we need to estimate the term
$$
\sup_{\lambda\in K} \tau_\lambda \left(  [\psi^{(\ell)}(e), x ]^*[\psi^{(\ell)}(e), x ] \right).
$$
If we think of $K$ as the union of the families $\U_1,\U_2,...,\U_{m+1}$, then the supremum can be taken only over $\U_\ell = \bigcup_{k=1}^{s(\ell)}O_k^{(\ell)}$. However, on each subset $O_k^{(\ell)}$ only the term with $g_k^{(\ell)}$ survives as $\U_\ell$ consists of disjoint subsets. Therefore we have
\begin{align*}
\| [\psi^{(\ell)}(e),x] \|_{2,u}^2 &= \sup_{1\leq k\leq s(\ell), \lambda\in O_k^{(\ell)}} \tau_\lambda \left( g_k^{(\ell)^{\frac{1}{2}}} [\varphi_k^{(\ell)}(e), x ]^*  g_k^{(\ell)}  [\varphi_k^{(\ell)}(e), x ]  g_k^{(\ell)^{\frac{1}{2}}}  \right) \\
&\leq \sup_{1\leq k\leq s(\ell)}\left(  \sup_{\lambda\in O_k^{(\ell)}} \left\|   [\varphi_k^{(\ell)}(e), x   ]  \right\|_{2,\lambda}^2   \right) \\
&< \e^2.
\end{align*}
\end{proof}

The following lemma gives a sufficient condition for a c.p.c map to be order zero. This is known among experts (cf. \cite[proof of Proposition 7.7]{KirRo01}), but for the reader's convenience we include a proof here. The argument is due to Aleksey Zelenberg. 

\begin{lem}\label{lem:orderzero}
Let $A$ and $B$ be unital $C^*$-algebras. Suppose $\varphi:A\to B$ is a c.p.c. map such that
$$
\varphi(1_A)\varphi(x^*x) = \varphi(x)^*\varphi(x)
$$
for all $x\in A$. Then $\varphi$ is an order zero map.
\end{lem}
\begin{proof}
Suppose $x,y$ are orthogonal positive elements in $A$ (namely $xy=0$). We compute
\begin{align*}
\varphi(1)\varphi((x+y)^*(x+y)) &= \varphi(1)\varphi(x^*x + y^*y) \\
&= \varphi(1)\varphi(x^2)+ \varphi(1)\varphi(y^2) \\
&= \varphi(x)^2 + \varphi(y)^2.
\end{align*}
On the other hand,
\begin{align*}
\varphi(x+y)^*\varphi(x+y) &= \varphi(x)^2 + \varphi(x)\varphi(y)+ \varphi(y)\varphi(x) + \varphi(y)^2.
\end{align*}
Comparing these two expressions (they are equal by the assumption), we see that
$$
\varphi(x)\varphi(y) + \varphi(y)\varphi(x) = 0.
$$
Write $a := \varphi(x)$ and $b:=\varphi(y)$. To show that $ab = 0$, it suffices to show that $a$ and $b$ commute. From the identity above ($ab+ba = 0$) we have
$$
ab^2 = (ab)b = (-ba)b = -b(ab) = -b(-ba) = b^2a.
$$
This shows that $a$ commutes with $b^2$, which implies that $a$ commutes with $b$.
\end{proof}

Now we are ready to prove Proposition \ref{prop:fixpointorderzero}.

\begin{proof} (of Proposition \ref{prop:fixpointorderzero})
Let $\{x_1,x_2,...\}$ be a countable $\|\cdot\|_{2,u}$-dense subset of $\hm$. By Lemma \ref{lem:fixpointcpc}, for each $n\in \N$ there exist c.p.c. maps $\psi_n^{(1)}, \psi_n^{(2)},...,\psi_n^{(m+1)}:M_p\to \hm$ such that
(ii)-(v) in Lemma \ref{lem:fixpointcpc} hold with $\e = \frac{1}{n}$ and
\begin{enumerate}
\item[(i)'] $\| [\psi_n^{(\ell)}(e), x_j] \|_{2,u} < \frac{1}{n}\;\;\;\;\;\; (1\leq \ell\leq m+1,\; e\in (M_p)_1,\; 1\leq j\leq n)$.
\end{enumerate}
For each $1\leq \ell\leq m+1$ define c.p.c. maps $\psi_\ell:M_p\to \hm^{\omega}$ by
$$
\psi_\ell(e) = \pi_\omega(\psi_1^{(\ell)}(e), \psi_2^{(\ell)}(e), \psi_3^{(\ell)}(e),... ),
$$
where $\pi_\omega$ is the quotient map from $\ell^\infty(\N,\hm)$ onto $\hm^\omega$. Then by construction 
\begin{itemize}
\item the image of $\psi_\ell$ belongs to $(\hm^\omega\cap \hm')^{\beta_\omega}$;
\item $\psi_1, \psi_2,...,\psi_{m+1}$ have commuting ranges;
\item $\psi_1(1)+\cdots + \psi_{m+1}(1) =1$;
\item $\psi_\ell(1)\psi_{\ell}(e^*e) = \psi_\ell(e)^*\psi_\ell(e)$\;\;\;\;\;\; $(1\leq \ell\leq m+1,\; e\in M_p)$.
\end{itemize}
It remains to show that each $\psi_\ell$ is order zero, but this follows from the fourth bullet, according to Lemma \ref{lem:orderzero}.
\end{proof}


\section{A Rokhlin type theorem for continuous $W^*$-bundles}

In this section we present a Rokhlin type theorem for pointwise outer $\Z$-actions on a continuous $W^*$-bundle. More precisely we have the following:

\begin{thm}
\label{thm:bundleRokhlin}
Let $(\hm, K, E)$ be a separable continuous $W^*$-bundle such that $\pi_\lambda(\hm)\cong \hr$ for all $\lambda\in K$, and let $(\beta, \id)$ be an automorphism of $(\hm,K,E)$. Suppose the fiber map $\beta_\lambda\in \Aut(\pi_\lambda(\hm))$ is pointwise outer for each $\lambda\in K$ and suppose for each $p\in \N$ there is a unital *-homomorphism $\psi:M_p\to (\hm^\omega\cap \hm')^{\beta_\omega}$. Then for every $n\in \N$, there exist projections $p_1,...,p_n\in \hm^{\omega}\cap \hm'$ such that 
\begin{enumerate}
\item[(i)] $\beta_\omega(p_k) = p_{k+1} \mod n$\;\;\;\;\;\;\;\; $(1\leq k\leq n)$
\item[(ii)] $\sum_{k=1}^n p_k = 1$.
\item[(iii)] $\left\| 1-p_1 \right\|_{2,u}^{(\omega)} < 1$.
\end{enumerate}
\end{thm} 

Again we will divide the proof into smaller pieces. The first lemma is known to experts and follows from the uniqueness of the trace on a matrix algebra.

\begin{lem}
\label{lem:matrixtrace}
Let $A$ be a unital $C^*$-algebra. Suppose there is a unital *-homomorphism $\sigma:M_p\to A$. Then for each $a\in \sigma(M_p)$, $b\in A\cap \sigma(M_p)'$, and each trace $\tau\in T(A)$, we have $\tau(ab) = \tau(a)\tau(b)$.
\end{lem}
\begin{proof}
By the standard decomposition we can reduce to the case that $0\leq b\leq 1$. If $\tau(b) = 0$, then
$$
\tau(ab) = \tau(a b^{1/2}b^{1/2}) \leq \tau(b^{1/2}a^*ab^{1/2})^{1/2} \tau(b)^{1/2} = 0.
$$
Hence the equality holds. Now assume that $\tau(b) > 0$. Note that the map
$$
M_p\to \C,\;\;\;\;\;\;\;\;\; e\mapsto \frac{1}{\tau(b)}\tau(\sigma(e)b)
$$
is a tracial state on $M_p$ because $\sigma(e)$ and $b$ commute. Since $M_p$ admits a unique trace, we have
$$
\tau \circ \sigma (e) = \frac{1}{\tau(b)} \tau( \sigma(e)b)
$$
and the lemma is proved.
\end{proof}

The first objective is to obtain a ``mutually orthogonal'' partition of unity that simulates what an ordinary partition of unity does. This is accomplished in Lemma \ref{lem:bundlepartition}, and the following lemma serves as preparation. We shall write $tr$ for the normalized trace on a matrix algebra.

\begin{lem}
\label{lem:bundleucp}
Let $(\hm, K, E)$ be a separable continuous $W^*$-bundle and let $(\beta,\sigma)$ be an automorphism of $(\hm, K, E)$. Suppose there exists a unital *-homomorphism $\psi:M_p\to (\hm^\omega\cap \hm')^{\beta_\omega}$. Then given $\e > 0$ and a norm bounded $\|\cdot\|_{2,u}$-compact subset $\Omega$ of $\hm$, there exists a u.c.p. map $\varphi:M_p\to\hm$ such that
\begin{enumerate}
\item[(i)] $\| [ \varphi(e), x] \|_{2,u} < \e \;\;\;\;\;\; (e\in (M_p)_1, x\in \Omega);$
\item[(ii)] $\| \beta(\varphi(e))-\varphi(e)\|_{2,u} < \e\;\;\;\;\;\; (e\in (M_p)_1);$
\item[(iii)] $\| \varphi(ef)-\varphi(e)\varphi(f)\|_{2,u} < \e\;\;\;\;\;\; (e,f\in (M_p)_1);$
\item[(iv)] $\|[ \varphi(e)^{1/2}, x] \|_{2,u} < \e\;\;\;\;\;\; (e\in (M_p)_1, x\in \Omega);$
\item[(v)] $\left|  \tau_\lambda(\varphi(e)x) - \tau_\lambda(\varphi(e))\tau_\lambda(x)  \right| < \e\;\;\;\;\;\; (e\in (M_p)_1, x\in \Omega, \lambda\in K)$;
\item[(vi)] $\left| \tau_\lambda(\varphi(e))-\tr(e)\right| < \e\;\;\;\;\;\; (e\in (M_p)_1, \lambda\in K)$.
\end{enumerate}
\end{lem}
\begin{proof}
Lift $\psi$ to a sequence of u.c.p. maps $\varphi_n:M_p\to \hm$. Then we have
\begin{enumerate}
\item[(a)] $\lim_{n\to\omega}\| [\varphi_n(e),x]\|_{2,u} = 0\;\;\;\;\;\; (e\in M_p, x\in \hm);$
\item[(b)] $\lim_{n\to\omega} \| \beta(\varphi_n(e)) - \varphi_n(e) \|_{2,u} = 0\;\;\;\;\;\; (e\in M_p);$
\item[(c)] $\lim_{n\to \omega} \| \varphi_n(ef)-\varphi_n(e)\varphi_n(f) \|_{2,u} = 0\;\;\;\;\;\; (e,f\in M_p);$
\item[(d)] $\lim_{n\to\omega} \| [\varphi_n(e)^{1/2}, x] \|_{2,u} = 0\;\;\;\;\;\; (e\in M_p^{+}, x\in \hm).$
\end{enumerate}
By Lemma \ref{lem:matrixtrace}, for each tracial state $\tau_\omega$ on $\hm^{\omega}$ we have
$$
\tau_\omega(\psi(e)x) =  \tau_\omega(\psi(e))\tau_\omega(x)\;\;\;\;\;\; (e\in M_p, x\in \hm).
$$
Also since $M_p$ has a unique trace, $\tau_\omega \circ \psi = \tr$ on $M_p$. In particular,
\begin{enumerate}
\item[(e)] $\lim_{n\to\omega} \left| \tau_\lambda(\varphi_n(e)x) - \tau_\lambda (\varphi_n(e))\tau_\lambda(x) \right| = 0\;\;\;\;\;\; (e\in M_p, x\in \hm, \lambda\in K)$;
\item[(f)] $\lim_{n\to\omega} \left| \tau_\lambda(\varphi_n(e)) - \tr(e)  \right| = 0\;\;\;\;\;\;\; (e\in M_p, \lambda\in K)$.
\end{enumerate}
Since $(M_p)_1$, $\Omega$, and $K$ are compact, the lemma is proved by choosing $n$ sufficiently large along $\omega$.
\end{proof}

\begin{lem}
\label{lem:bundlepartition}
Let $(\hm,K,E)$ be a separable continuous $W^*$-bundle, and let $(\beta,\id)$ be an automorphism of $(\hm,K,E)$. Suppose for each $p\in \N$ there exists a unital *-homomorphism $\psi: M_p\to (\hm^\omega\cap \hm')^{\beta_\omega}$. Then given $\e > 0$, a norm bounded $\|\cdot\|_{2,u}$-compact subset $\Omega$ of $\hm$, and a partition of unity $\{g_j\}_{j=1}^N$ subordinate to some open cover of $K$, there exist positive contractions $a_1,a_2,...,a_N\in \hm$ such that 
\begin{enumerate}
\item[(i)] $\| [a_j,x] \|_{2,u} < \e\;\;\;\;\;\; (1\leq j\leq N, x\in \Omega)$;
\item[(ii)] $\| \beta(a_j)-a_j\|_{2,u} < \e\;\;\;\;\;\; (1\leq j\leq N)$;
\item[(iii)] $\| a_ia_j\|_{2,u} < \e\;\;\;\;\;\; (1\leq i\neq j\leq N)$;
\item[(iv)] $\| [a_j^{1/2}, x] \|_{2,u} < \e\;\;\;\;\;\;\; (1\leq j\leq N, x\in \Omega)$;
\item[(v)] $\| a_i^{1/2}a_j^{1/2} \|_{2,u} < \e\;\;\;\;\;\;\; (1\leq i\neq j \leq N)$;
\item[(vi)] $\left| \tau_\lambda(a_jx) - g_j(\lambda)\tau_\lambda(x) \right| < \e\;\;\;\;\;\; (1\leq j\leq N, x\in \Omega, \lambda\in K)$;
\item[(vii)] $\sum_{j=1}^N a_j = 1$.
\end{enumerate}
\end{lem}
\begin{proof}
Without losing generality we may assume that $\e \leq 1$ and that $\Omega$ is contained in the (norm) unit ball of $\hm$. Put $\e' := \frac{\e^2}{4}$. Choose $p\in \N$ sufficiently large so that $\frac{1}{p} < \e'$. By the previous lemma there is a u.c.p. map $\varphi:M_p\to \hm$ such that (i)-(vi) in that lemma hold for $\e'$ and $\Omega$. 

Now we apply a technique from \cite{Oz01}. Let $\lfloor \cdot \rfloor$ denote the floor function. For each $t\in [0,p]$, define a positive contraction $p_t\in M_p$ by
$$
p_t = \diag(1,...,1, t - \lfloor t \rfloor,0,...,0),
$$
where $t- \lfloor t \rfloor$ is in the $(\lfloor t \rfloor + 1)$-th position. Note that the (unnormalized) trace of $p_t$ is equal to $t$. For $0\leq s\leq t\leq p$, we write
$$
p_{[s,t]} := p_t - p_s.
$$
Consider the functions $h_0,h_1,...,h_N\in C(K)$ defined by $h_0 \equiv 0$ and $h_j = p\cdot \sum_{i=1}^j g_i$ for $1\leq j\leq N$. In other words, $h_{j} - h_{j-1} = p\cdot g_j$ for $1\leq j\leq N$. Define, for $1\leq j\leq N$ and $\lambda\in K$,
$$
p_j(\lambda) := p_{[h_{j-1}(\lambda), h_j(\lambda)  ]}\in M_p
$$
and
$$
a_j(\lambda) = \varphi(p_j(\lambda))\in\hm.
$$
Observe that $\tr(p_j(\lambda)) = g_j(\lambda)$ and $\tr(p_i(\lambda)p_j(\lambda)) \leq \frac{1}{4p}$ when $i\neq j$. Since the map $\lambda\mapsto a_j(\lambda)$ is norm continuous, by \cite[Theorem 11]{Oz01} for each $j$ there exists an element $a_j\in \hm$ such that
$\pi_\lambda(a_j) = \pi_\lambda(a_j(\lambda))$ for all $\lambda\in K$. We claim that the positive contractions $a_1,...,a_N$ satisfy the conditions (i)-(vii). We demonstrate how to verify conditions (i) and (iii). The other conditions can be checked in a similar way.

The estimate (i) follows from a simple calculation:
\begin{align*}
\| [ a_j, x] \|_{2,u} &= \sup_{\lambda\in K} \| [a_j,x]\|_{2,\lambda} \\
&= 
\sup_{\lambda\in K} \| [\pi_\lambda(a_j), \pi_\lambda(x)] \|_{2,\lambda} \\
&=
\sup_{\lambda\in K} \| [\pi_\lambda(a_j(\lambda)), \pi_\lambda(x)] \|_{2,\lambda} \\
&=
\sup_{\lambda\in K} \| [a_j(\lambda), x]\|_{2,\lambda}\\
&=
\sup_{\lambda\in K} \| [\varphi(p_j(\lambda)), x] \|_{2,\lambda} \\
&< \e' < \e \;\;\;\;\;\;\; (\text{Lemma \ref{lem:bundleucp} (i)}).
\end{align*}

To estimate (iii), we first make the following observation. For $1\leq i\neq j\leq N$ and $\lambda\in K$, we have
\begin{align*}
\tau_\lambda(a_ia_j) &= \ip{ \pi_\lambda(a_i)\pi_\lambda(a_j)\hat{1},\hat{1}} \\
&=
\ip{ \pi_\lambda(a_i(\lambda))\pi_\lambda(a_j(\lambda))\hat{1}, \hat{1} } \\
&=
\tau_\lambda(a_i(\lambda)a_j(\lambda)) \\
&=
\tau_\lambda( [\varphi(p_i(\lambda))][\varphi(p_j(\lambda))] ). \\
\end{align*}
Since $\varphi$ is almost multiplicative with respect to the uniform 2-norm (Lemma \ref{lem:bundleucp} (iii)), the last expression is almost the same as $\tau_\lambda \circ\varphi(p_i(\lambda)p_j(\lambda  ))$, which in turn is almost the same as $\tr( p_i(\lambda)p_j(\lambda  )   )$ by Lemma \ref{lem:bundleucp} (vi). Recall that $p_i(\lambda)$ and $p_j(\lambda)$ are almost orthogonal with respect to the trace. Therefore we have
\begin{align*}
\tau_\lambda(a_ia_j) < \tr( p_i(\lambda)p_j(\lambda  )   ) + 2\e' < 3\e'.
\end{align*}
Now we compute
\begin{align*}
\| a_ia_j \|_{2,u}^2 &= \sup_{\lambda\in K} \tau_\lambda(a_j a_i^2 a_j) \\
&\leq 
\sup_{\lambda\in K} \tau_\lambda(a_ja_ia_j) \\
&= 
\sup_{\lambda\in K}\tau_\lambda(a_i^{1/2}a_j^2a_i^{1/2}) \\
&\leq
\sup_{\lambda\in K} \tau_\lambda(a_ia_j) < 3\e',
\end{align*}
and hence 
$$
\|a_ia_j\|_{2,u} < \sqrt{3\e'}  < \e.
$$
\end{proof}

\begin{rem}
In Lemma \ref{lem:bundlepartition}, the assumption that the automorphism is trivial on $K$ is needed when one verifies condition (ii). More precisely one needs the identity $\beta_\lambda(\pi_\lambda(a_j(\lambda))) = \pi_\lambda(\beta(a_j(\lambda)))$.
\end{rem}

With a mutually orthogonal partition of unity at our disposal, we can now glue the Rokhlin towers in the fibers without destroying the orthogonality. In this way we obtain a global Rokhlin tower as claimed in Theorem \ref{thm:bundleRokhlin}.

\begin{proof} (of Theorem \ref{thm:bundleRokhlin}) We will show that given $n\in \N$, $\e > 0$, and a norm-bounded $\|\cdot\|_{2,u}$-compact subset $\Omega$ of $\hm$, there exist positive contractions $r_1, r_2,...,r_n$ in $\hm$ such that
\begin{enumerate}
\item[(a)] $\| \beta(r_k)-r_{k+1}\|_{2,u} < \e\mod n\;\;\;\;\;\;\;\; (1\leq k\leq n)$;
\item[(b)] $\| r_kr_\ell \|_{2,u} < \e\;\;\;\;\;\;\;\; (1\leq k\neq \ell\leq n)$;
\item[(c)] $\| [r_k,x] \|_{2,u} < \e\;\;\;\;\;\;\;\; (1\leq k\leq n; x\in \Omega)$;
\item[(d)] $\left\| \sum_{k=1}^n r_k - 1_\hm  \right\|_{2,u} < \e$;
\item[(e)] $\sup_{\lambda\in K} \left| \tau_\lambda(r_1) - \frac{1}{n}  \right| < \e$.
\end{enumerate}
Once this is established, a standard argument for ultraproducts shows that we can find projections $p_1,...,p_n\in \hm^\omega\cap \hm'$ such that
\begin{enumerate}
\item[(i)] $\beta_\omega(p_k) = p_{k+1} \mod n\;\;\;\;\;\;\;\; (1\leq k\leq n)$;
\item[(ii)] $\sum_{k=1}^n p_k = 1$.
\end{enumerate}
By $(e)$ we may arrange that $\tau_{\lambda_\omega}(p_1) = \lim_{n\to\omega}\tau_{\lambda_n}(p_{1,n}) = \frac{1}{n}$, where $(p_{1,n})$ is a representing sequence of $p_1$ and $\lambda_\omega = (\lambda_1,\lambda_2,...)$ is an element in the set-theoretic ultraproduct $K_\omega$ of $K$. Since $K_\omega$ is dense in the ultracoproduct $K^\omega$, by the 2-norm continuity we actually have $\tau_{\lambda_\omega}(p_1) = \frac{1}{n}$ for all $\lambda_\omega$ in $K^\omega$. It follows that
\begin{align*}
\left( \| 1- p_1\|_{2,u}^{(\omega)}\right)^2 &= \left\| E^\omega(1-p_1)^*(1-p_1)  \right\| \\
&= \left\| E^\omega(1-p_1)  \right\| \\
&= \sup_{\lambda_\omega\in K^\omega}\tau_{\lambda_\omega}(1-p_1) \\
&= 1-\frac{1}{n} < 1
\end{align*}
and hence $\|1-p_1\|_{2,u}^{(\omega)} < 1$.

We make one more remark before we begin the construction of the postive contractions $r_1,r_2,...,r_n$ in $\hm$. Under the current assumption (namely the automorphism is trivial on $K$), condition (e) really is a consequence of (a) and (d). First note that by (a) for each $\lambda\in K$ and $1\leq k\leq n$,
$$
\tau_\lambda(r_k) = ev_\lambda\circ E(r_k) = ev_\lambda\circ E\circ \beta(r_k) = \tau_\lambda(\beta(r_k)) \approx \tau_\lambda(r_{k+1}).
$$
Therefore
$$
\left| n\tau_\lambda(r_1) - 1  \right| \approx \left| \tau_\lambda\left( \sum_{k=1}^n r_k  \right) - 1  \right| = \left| \tau_\lambda\left( \sum_{k=1}^n r_k - 1  \right) \right|,
$$
and the last expression is small by condition (d).

Now we start to construct $r_1,r_2,..,r_n$ in $\hm$ as claimed. We may assume that $\e\leq 1$ and that $\Omega$ is contained in the norm-closed unit ball of $\hm$. Put $\e' = \frac{\e}{4}$. By Connes' noncommutative Rokhlin theorem (\cite{Con01}), for each $\lambda\in K$ we can find (mutually orthogonal) projections $p_{\lambda,1},p_{\lambda,2},...,p_{\lambda,n}$ in $\pi_\lambda(\hm)$ such that 
\begin{enumerate}
\item[(a)'] $\| \beta_\lambda(p_{\lambda,k})-p_{\lambda,k+1} \|_{2,\lambda} < \e' \mod n \;\;\;\;\;\;\;\; (1\leq k\leq n)$; 
\item[(c)'] $\| [p_{\lambda,k}, \pi_\lambda(x)]\|_{2,\lambda} < \e'\;\;\;\;\;\;\;\; (1\leq k\leq n; x\in \Omega)$;
\item[(d)'] $\sum_{k=1}^n p_{\lambda, k} = 1_{\pi_\lambda
(\hm)}$.
\end{enumerate}
Fix $\lambda\in K$ for a while. Lift $p_{\lambda,1},...,p_{\lambda,n}$ along $\pi_\lambda:\hm\to \pi_\lambda(\hm)$ to positive contractions $r_{\lambda,1},...,r_{\lambda,n}$ in $\hm$. By the 2-norm continuity of elements in $\hm$, there exists a neighborhood $O_\lambda$ of $\lambda$ such that
\begin{enumerate}
\item[(a)''] $\sup_{\mu\in O_\lambda}\| \beta(r_{\lambda,k}) - r_{\lambda,k+1}\|_{2,\mu} < \e' \mod n \;\;\;\;\;\;\;\; (1\leq k\leq n)$;
\item[(b)''] $\sup_{\mu\in O_\lambda} \| r_{\lambda,k} r_{\lambda,\ell} \|_{2,\mu} < \e'\;\;\;\;\;\;\;\; (1\leq k\neq \ell\leq n)$.
\item[(c)''] $\sup_{\mu\in O_\lambda} \| [r_{\lambda,k}, x] \|_{2,\mu} < \e'\;\;\;\;\;\;\;\;\ (1\leq k\leq n; x\in \Omega)$;
\item[(d)''] $\sup_{\mu\in O_\lambda}\left\| \sum_{k=1}^n r_{\lambda,k} - 1_\hm  \right\|_{2,\mu} < \e'$;
\end{enumerate}

Since $K$ is compact, we can extract a finite subcover $O_1,...,O_N$ with corresponding elements $\{r_{1,1},...,r_{1,n}\}$,...,$\{ r_{N,1},...,r_{N,n} \}$. Let $\{g_j\}_{j=1}^N$ be a partition of unity subordinate to the open cover $\bigcup_{j=1}^N O_j$. Put
\begin{align*}
S := &\{r_{j,k}\} \cup \{r_{j,k}r_{j,\ell}\} \cup \{ \beta(r_{j,k})-r_{j,k+1} \} \cup \{ [r_{j,k}, x] \}  \cup \left\lbrace \sum_{j,k} r_{j,k} - 1  \right\rbrace
\end{align*}
(where the union is taken over $1\leq j\leq N, 1\leq k, \ell \leq N, x\in \Omega$),
\begin{align*}
\Omega'' &:= \Omega\cup S\cup S^*\cup SS^*\cup S^*S,
\end{align*}
and
$$
\e'' = \min \left\{  \frac{\e^2}{16(4n^2N^2+3N)},\;\; \frac{\e}{2(N^2+3N)}  \right\}.
$$
By Lemma \ref{lem:bundlepartition} we can find positive contractions $a_1,...,a_N\in \hm$ satisfying (i)-(vii) in the same lemma, with $\Omega$ and $\e$ replaced by $\Omega''$ and $\e''$, respectively. Define, for $1\leq k\leq n$,
$$
r_k = \sum_{j=1}^N a_j^{1/2}r_{j,k}a_j^{1/2}\in \hm.
$$
Note that $r_k$ is a positive contraction since
$$
r_k =  \sum_{j=1}^N a_j^{1/2}r_{j,k}a_j^{1/2}  \leq  \sum_{j=1}^N a_j^{1/2}\| r_{j,k}\| a_j^{1/2} \leq  \sum_{j=1}^N a_j = 1.
$$
We claim that $r_1, r_2,..., r_n$ satisfy conditions (a)-(d). We shall only verify (a) since the other conditions can be handled in a similar way.

Given $k\in \{1,...,n\}$, we write
\begin{align*}
 \beta(r_k)-r_{k+1}  &=   \sum_{j=1}^N \beta( a_j^{1/2}r_{j,k}a_j^{1/2} ) - \sum_{j=1}^N a_j^{1/2}r_{j,k+1}a_j^{1/2}  \\
&=: 
I_1 + I_2 + I_3 + I_4,
\end{align*}
where
\begin{align*}
I_1 &:= \sum_{j=1}^N \beta( a_j^{1/2} [  r_{j,k}, a_j^{1/2}  ]    );\;\;\;\;\;\;\;\; 
I_2 := \sum_{j=1}^N [\beta(a_j)-a_j]\beta(r_{j,k})  ;\\
I_3 &:= \sum_{j=1}^N a_j[ \beta(r_{j,k}) - r_{j,k+1}  ];\;\;\;\;\;\;\;\;
I_4 := \sum_{j=1}^N a_j^{1/2} [ a_j^{1/2}, r_{j,k+1}  ].
\end{align*}
It is easy to see that $\|I_1\|_{2,u}$, $\|I_2\|_{2,u}$, and $\|I_4\|_{2,u}$ are all bounded by $N\e''$, using Lemma \ref{lem:bundlepartition} (ii) and (iv). Before we estimate $\|I_3\|_{2,u}$, we make a simple but useful calculation. For each $\lambda\in K$, $y,z \in\hm$, and $1\leq i\neq j\leq N$, we have
\begin{align*}
\left|  \tau_\lambda(y^*a_ja_iz)  \right| &= |\tau_\lambda(zy^*a_ja_i) | \\
&\leq 
\|zy^*\|_{2,\lambda} \|a_ja_i\|_{2,\lambda} \\
&\leq 
\| y\| \|z\| \|a_ja_i\|_{2,\lambda}.
\end{align*}
To estimate $I_3$, we first write
\begin{align*}
I_3^*I_3 &=  \sum_{i,j=1}^N [ \beta(r_{j,k}) - r_{j,k+1} ]^* a_ja_i [ \beta(r_{i,k}) - r_{i,k+1} ]  \\
&=: J_1 + J_2,
\end{align*}
where $J_1$ consists of the ``cross terms`` and $J_2$ consists of the ``diagonal terms''. More precisely,
\begin{align*}
J_1 &:= \sum_{i\neq j}   [ \beta(r_{j,k}) - r_{j,k+1} ]^* a_ja_i [ \beta(r_{i,k}) - r_{i,k+1} ];\\
J_2 &:= \sum_{j=1}^N  [ \beta(r_{j,k}) - r_{j,k+1} ]^* a_j^2 [ \beta(r_{j,k}) - r_{j,k+1} ]. 
\end{align*}
Since $\|a_ja_i\|_{2,u} < \e''$, it is not hard to see that $\tau_\lambda(J_1)$ is bounded by $4N^2\e''$ using the inequality $|\tau_\lambda(y^*a_ja_iz)| \leq \|y\|\|z\|\|a_ja_i\|_{2,\lambda}$ established above. On the other hand,
\begin{align*} 
\tau_\lambda(J_2) &= \sum_{j=1}^N \tau_\lambda\left( [ \beta(r_{j,k}) - r_{j,k+1} ]^* a_j^2 [ \beta(r_{j,k}) - r_{j,k+1} ]    \right)  \\
&\leq
\sum_{j=1}^N \tau_\lambda\left(  a_j [ \beta(r_{j,k}) - r_{j,k+1} ][ \beta(r_{j,k}) - r_{j,k+1} ]^*   \right) \\
&\leq 
\sum_{j=1}^N g_j(\lambda) \tau_\lambda\left(  [ \beta(r_{j,k}) - r_{j,k+1} ][ \beta(r_{j,k}) - r_{j,k+1} ]^*  \right) + N\e''\;\;\;\;\;\;\;\; \text{(Lemma \ref{lem:bundlepartition} (vi))} \\
&=
\sum_{j=1}^N g_j(\lambda) \|  \beta(r_{j,k}) - r_{j,k+1} \|_{2,\lambda}^2 +N\e'' \\
&<
\sum_{j=1}^Ng_j(\lambda)\cdot (\e')^2 + N\e'' \\
&= (\e')^2 + N\e''.
\end{align*}
In summary, we have
$$
\|I_3\|_{2,u}^2 = \sup_{\lambda\in K}\tau_\lambda(I_3^*I_3) = \sup_{\lambda\in K}\tau_\lambda(J_1+J_2) < (\e')^2 + (4N^2+N)\e'',
$$
and hence
\begin{align*}
\| \beta(r_k) - \beta(r_{k+1}) \|_{2,u} &\leq \|I_1\|_{2,u} + \|I_2\|_{2,u} + \|I_3\|_{2,u} + \|I_4\|_{2,u} \\
&< 3N\e'' + \sqrt{ (\e')^2 + (4N^2+N)\e''} \\
&< \e. 
\end{align*}
When trying to verify conditions (b)-(d), one would encounter exactly the same type of estimates consisting of cross terms and diagonal terms. The cross terms are easy to bound because $\|a_ja_i\|_{2,u}$ is small when $j\neq i$, and the diagonal terms can be handled by taking out the partition of unity which localizes the estimate. We leave the details to the reader.
\end{proof}


\section{Approximate Rokhlin Property for $C^*$-algebras}

In this section we transfer the Rokhlin property for $W^*$-bundles to a ``Rokhlin property in trace'' for $C^*$-algebras. The main ingredient, as mentioned in the introduction, is the canonical surjection from $F(A)$ onto the central sequence of $W^*$-bundles. Here we recall the relevant definitions and results.

\begin{defn} (\cite[Definition 4.4]{KirRo01}) \label{def:sigmaideal}
Let $B$ be a $C^*$-algebra and let $J$ be a closed ideal in $B$. Then $J$ is a \textit{$\sigma$-ideal} of $B$ if for every separable $C^*$-subalgebra $C$ of $B$ there exists a positive contraction $e\in J$ such that $e\in C'\cap B$ and $ec = c$ for all $c\in C\cap J$.
\end{defn}

Let $A$ be a unital separable $C^*$-algebra such that $T(A) \neq \emptyset$. Recall that the trace-kernel ideal $J_A$ in $A_\omega$ is defined by
$$
J_A = \{ (a_n)_n\in A_\omega:\lim_{n\to\omega} \|a_n\|_{2,u} =0 \}.
$$

\begin{prop}\label{prop:traceideal} \cite[Proposition 4.6]{KirRo01} Let $A$ be a unital $C^*$-algebra with $T(A)\neq \emptyset$, then the trace-kernel ideal $J_A$ is a $\sigma$-ideal in $A_\omega$.
\end{prop}

The proof relies on two important techniques. The first one is the existence of a quasi-central approximate unit. In fact, one can think of the element $e$ in Definition \ref{def:sigmaideal} as a repackaging of a quasi-central approximate unit using the language of ultraproducts, This repackaging relies on the second technique - an ultraproduct version of the standard diagonal argument commonly known as ``Kirchberg's $\e$-test'':

\begin{lem}\label{lem:epsilontest} $($\cite[Lemma 3.1]{KirRo01}$)$ Let $\omega$ be a free ultrafilter on $\N$. Let $X_1,X_2,...$ be any sequence of sets and suppose that, for each $k\in \N$, we are given a sequence $(f_n^{(k)})_{n\geq 1}$ of functions $f_n^{(k)}:X_n\to [0,\infty)$.

For each $k\in \N$, define $f_\omega^{(k)}:\prod_{n=1}^\infty X_n\to [0,\infty]$ by
$$
f_\omega^{(k)}(s_1,s_2,...) = \lim_{n\to\omega}f_n^{(k)}(s_n),\;\;\;\;\;\;\;\; (s_n)_{n\geq 1}\in \prod_{n=1}^\infty X_n.
$$
Suppose for every $m\in \N$ and $\e > 0$, there exists $s = (s_1,s_2,...)\in \prod_{n=1}^\infty X_n$ such that $f_\omega^{(k)}(s) < \e$ for $1\leq k\leq m$. Then there exists $t=(t_1,t_2,...)\in \prod_{n=1}^\infty X_n$ with $f_\omega^{(k)}(t) = 0$ for all $k\in \N$.
\end{lem}

The inclusion $A\to \hm$ induces a well defined *-homomorphism from $A_\omega$ into $\hm^\omega$. Since $\hm$ is by definition the uniform 2-norm completion of $A$, there is a built-in Kaplansky type approximation and hence the map $A_\omega\to \hm^\omega$ is surjective (cf. \cite[Lemma 3.10]{BBSTWW}). Now the surjectivity of the canonical map $F(A)\to \hm^\omega\cap \hm'$ is a simple corollary of Proposition \ref{prop:traceideal}. For a proof we refer the reader to \cite[Remark 4.7]{KirRo01}.

\begin{cor}\label{cor:KRsurjectivity} Let $A$ be a unital separable $C^*$-algebra such that $T(A)\neq \emptyset$ and $\partial T(A)$ is compact. Then the canonical map $F(A)\to \hm^\omega\cap \hm'$ is surjective.
\end{cor}

We shall need one more ingredient in order to lift a Rokhlin tower from the $W^*$-bundle back to the $C^*$-algebra. The following definition and proposition, which mimic Definition \ref{def:sigmaideal} and Proposition \ref{prop:traceideal}, are trying to capture the notion of an ``invariant'' quasi-central approximate unit in terms of ultraproducts.

\begin{defn}
Let $B$ be a $C^*$-algebra and $\alpha$ be an automorphism of $B$. Suppose $J$ is a closed ideal in $B$ satisfying $\alpha(J) = J$. Then $J$ is an \textit{$\alpha$-$\sigma$-ideal} if for every separable $C^*$-subalgebra $C$ of $B$ satisfying $\alpha(C) = C$, there exists a positive contraction $u$ in $C'\cap J$ such that
\begin{enumerate}
\item[(i)] $\alpha(u) = u$;
\item[(ii)] $uc = c$ for every $c\in C\cap J$.
\end{enumerate}
\end{defn}

\begin{prop}
\label{prop:alphasigmaideal}
Let $A$ be a unital separable $C^*$-algebra with $T(A)\neq \emptyset$, and let $\alpha$ be an automorphism of $A$. Then the trace-kernel ideal $J_A$ in $A_\omega$ is an $\alpha_\omega$-$\sigma$-ideal. 
\end{prop}
\begin{proof}
First note that $\alpha_\omega(J_A) = J_A$ since $\alpha_\omega$ is norm-preserving. Let $C$ be a separable $C^*$-subalgebra of $A_\omega$ satisfying $\alpha_\omega(C) = C$. By Proposition \ref{prop:traceideal}, $J_A$ is a $\sigma$-ideal in $A_\omega$, which means there exists a positive contraction $e\in C'\cap J_A$ such that $ec = c$ for all $c\in C\cap J_A$. For each $n\in \N$ define
$$
e^{(n)} := \frac{1}{n}\sum_{k=0}^{n-1}\alpha_\omega^k(e).
$$
Observe that each $e^{(n)}$ enjoys the same property as $e$. Indeed, for $c\in C$ we have
$$
\left\| [e^{(n)}, c] \right\| \leq \frac{1}{n}\sum_{k=0}^{n-1}\left\|  [\alpha_\omega^k(e), c]  \right\| = \frac{1}{n}\sum_{k=0}^{n-1}\left\|  [e, \alpha_\omega^{-k}(c)]  \right\| = 0,
$$
and for $c\in C\cap J_A$
$$
\| e^{(n)}c-c\|\leq \frac{1}{n}\sum_{k=1}^{n-1}\left\| \alpha_\omega^{k}(e)c-c   \right\| = \frac{1}{n}\sum_{k=1}^{n-1}\left\| e\alpha_\omega^{-k}(c)-\alpha_\omega^{-k}(c)   \right\| =0.
$$
Moreover,
$$
\left\| \alpha_\omega(e^{(n)}) - e^{(n)} \right\| \leq \frac{1}{n}\left\| \alpha_\omega^n(e) - e  \right\| \leq \frac{2}{n}.
$$

Let $d=(d_1,d_2,...)$ be a strictly positive contraction in the separable $C^*$-algebra $C\cap J_A$. Take a dense sequence $\{c^{(k)}\}_{k=1}^\infty$ in the unit ball of $C$, and represent each $c^{(k)}$ by $(c_1^{(k)}, c_2^{(k)},...)$. Let each $X_n$ be the set of positive contractions in $A$, and define functions $f_n^{(k)}:X_n\to [0,\infty)$ by
\begin{align*}
f_n^{(1)}(x) &= \|(1-x)d_n\|; \\
f_n^{(2)}(x) &= \|x\|_2;\\
f_n^{(3)}(x) &= \| \alpha(x) -x\|;\\
f_n^{(k+3)}(x) &= \|c_n^{(k)}x - xc_n^{(k)} \|\;\;\;\;\;\;\;\; (k\geq 1).
\end{align*}
Note that given $m\in \N$ and $\e >0$, there exists $\ell\in \N$ sufficiently large so that $e^{(\ell)}=(e_1^{(\ell)}, e_2^{(\ell)},...  )$ in $A_\omega$ satisfies
$$
f_\omega^{(k)}(e_1^{(\ell)},e_2^{(\ell)},...  ) < \e
$$
for all $1\leq k \leq m$. By the $\e$-test (Lemma \ref{lem:epsilontest}) there exists a sequence $(u_1,u_2,...)$ of positive contractions in $A$ such that
$$
f_\omega^{(k)}(u_1,u_2,...) = 0
$$
for all $k\in \N$. Then $u = (u_1,u_2,...)\in A_\omega$ is a positive contraction with the desired properties.
\end{proof}

There are various definitions of ``Rokhlin property in trace'' in the literature (e.g. \cite[Definition 4.2]{Ki02}, \cite[Definition 1.1]{OsPh01}, and \cite[Definition 1.1]{Sa01}). The following definition is similar, and it serves as a technical device that allows us to pass to finite Rokhlin dimension (see Section 6). Recall that if $(a_1,a_2,...)$ is a representative of an element $a$ in $A_\omega$, then we define
$$
\|a\|_{2,\omega} := \lim_{n\to\omega}\sup_{\tau\in T(A)} \tau(a_n^*a_n)^{1/2}.
$$
In other words, if $\pi:A_\omega\to \hm^\omega$ is the canonical map, then $\|a\|_{2,\omega} = \| \pi(a) \|_{2,u}^{(\omega)}$.

\begin{defn}
\label{defn:approxRokhlin}
Let $A$ be a unital separable stably finite exact $C^*$-algebra, and let $\alpha$ be an automorphism of $A$. We say $\alpha$ has the \textit{approximate Rokhlin property} if for every $n\in \N$ there exist positive contractions $f_1,...,f_n$ in $F(A)$ such that
\begin{enumerate}
\item[(i)] $\alpha_\omega(f_k) = f_{k+1}\mod n\;\;\;\;\;\;\;\; (1\leq k\leq n)$;
\item[(ii)] $f_k f_\ell = 0\;\;\;\;\;\;\;\; (1\leq k\neq \ell \leq n)$;
\item[(iii)] $1 - \sum_{k=1}^n f_k\in J_A\cap F(A)$.
\item[(iv)] $\sup_{m\in \N} \| 1-f_1^m\|_{2,\omega} < 1$.
\end{enumerate}
\end{defn}

Let $A$ be a unital separable $C^*$-algebra such that $T(A)\neq \emptyset$ and $\partial T(A)$ is compact. Given an automorphism $\alpha\in \Aut(A)$, there is an induced automorphism $(\tilde{\alpha},\bar{\alpha})$ of $(\hm, K, E)$ defined in an obvious way. This in turn induces an automorphism $(\tilde{\alpha}_\omega, \bar{\alpha}_\omega)$, as discussed in Section 3, of the ultraproduct $(\hm^\omega, K^\omega, E^\omega)$.

\begin{prop}
\label{prop:approxRokhlin}
Let $A$ be a unital separable $C^*$-algebra such that $T(A) \neq \emptyset$ and $\partial T(A)$ is compact, and let $\alpha$ be an automorphism of $A$. Suppose the canonical extension $(\tilde{\alpha}, \bar{\alpha})\in \Aut(\overline{A}^u, \partial T(A), E) =: \Aut(\hm, K, E)$ has the following Rokhlin property: for every $n\in \N$, there exist projections $p_1,...,p_n$ in $\hm^\omega\cap \hm'$ such that
\begin{enumerate}
\item[(i)] $\tilde{\alpha}_\omega(p_k) = p_{k+1}\mod n\;\;\;\;\;\;\;\; (1\leq k\leq n)$;
\item[(ii)] $\sum_{k=1}^n p_k = 1_\hm$;
\item[(iii)] $\| 1- p_1\|_{2,u}^{(\omega)} < 1$.
\end{enumerate}
Then $\alpha$ has the approximate Rokhlin property.
\end{prop}
\begin{proof}
Fix $n\in \N$. Find projections $p_1,...,p_n\in \hm^\omega\cap \hm'$ as in the assumption. Since the canonical map $F(A)\to \hm^\omega\cap \hm'$ is surjective, we can lift $\{p_k\}_{k=1}^n$ to positive contractions $\{f_k'\}_{k=1}^n$ in $F(A)$. Observe that
\begin{enumerate}
\item[(a)] $\alpha_\omega(f_k') - f_{k+1}'\in J_A\cap F(A)\mod n\;\;\;\;\;\;\;\ (1\leq k\leq n)$;
\item[(b)] $f_k'f_\ell'\in J_A\cap F(A)\;\;\;\;\;\;\; (1\leq k\neq \ell \leq n)$;
\item[(c)] $1-\sum_{k=1}^n f_k' \in J_A\cap F(A)$.
\end{enumerate}
If we put $C:= C^*(A, \{ \alpha_\omega^j(f_k') \}_{1\leq k\leq n}^{j\in \Z}  )$, then $C$ is a separable $C^*$-subalgebra of $A_\omega$ satisfying $\alpha_\omega(C) = C$. Since $J_A$ is an $\alpha_\omega$-$\sigma$-ideal in $A_\omega$ (Proposition \ref{prop:alphasigmaideal}), there exists a positive contraction $u\in C'\cap J_A$ such that
\begin{enumerate}
\item[(1)] $\alpha_\omega(u) = u$;
\item[(2)] $uc = c$ for every $c\in C\cap J_A$.
\end{enumerate}
Define, for $1\leq k\leq n$,
$$
f_k := (1-u)f_k'(1-u). 
$$
Then each $f_k$ is a positive contraction in $F(A)$ since both $f_k'$ and $u$ commute with $A$. Moreover $f_k$ is also a lift of $p_k$. We claim that $\{f_k\}_{1\leq k\leq n}$ satisfy the conditions for the approximate Rokhlin property. Indeed, since $\alpha_\omega(f_k')-f_{k+1}'$ belongs to $C'\cap J_A$, 
\begin{align*}
\alpha_\omega(f_k) - f_{k+1} &= \alpha_\omega((1-u)f_k'(1-u)) - (1-u)f_{k+1}'(1-u) \\
&= (1-u)\left[ \alpha_\omega (f_k') - f_{k+1}'  \right] (1-u) \\
&= 0.
\end{align*}
Similarly,
\begin{align*}
f_kf_\ell &= (1-u)f_k'(1-u)^2 f_\ell'(1-u) \\
&= (1-u)^3 f_k'f_\ell' (1-u) \\
&= 0.
\end{align*}
Observe that
\begin{align*}
1 -\sum_{k=1}^n f_k &= 1 - (1-u)\left( \sum_{k=1}^n f_k' \right)(1-u) \\
&= \left( 1 - \sum_{k=1}^n f_k'\right) + \sum_{k=1}^n (uf_k' + f_k' u - uf_k'u).
\end{align*}
Since $u\in J_A$ and $J_A$ is an ideal, the last expression belongs to $J_A\cap F(A)$. Finally,
$$
\sup_{m\in \N}\| 1- f_1^m \|_{2,\omega} = \sup_{m\in \N}\| 1 - \pi(f_1)^m \|_{2,u}^{(\omega)} = \sup_{m\in \N}\| 1 - p_1^m \|_{2,u}^{(\omega)} = \| 1 - p_1 \|_{2,u}^{(\omega)} < 1.
$$
This completes the proof.
\end{proof}


\section{Property (SI) and Rokhlin dimension}

In this section we prove Theorem \ref{thm:main} and Corollary \ref{cor:main}. In particular we discuss how to pass from the approximate Rokhlin property, introduced in the previous section, to Rokhlin dimension. The results described in this section are largely due to Matui and Sato. We will present their proof in a way that fits the language of this paper, but we don't claim any credit or originality.
 
As we will see, property (SI) plays a crucial role. Here we state the definition in terms of central sequence algebras as in \cite{KirRo01}. 

\begin{defn} (\cite[Definition 2.6]{KirRo01})
A unital simple $C^*$-algebra has \textit{property (SI)} if for every positive contraction $e,f\in F(A)$, with $e\in J_A$ and $\sup_n\| 1 - f^n \|_{2,\omega} < 1$, there exists $s\in F(A)$ with $fs = s$ and $s^*s = e$.
\end{defn}

\begin{prop}
\label{prop:mergebySI}
Let $A$ be a unital simple separable $C^*$-algebra such that $T(A)\neq \emptyset$ and $\partial T(A)$ is compact. and let $\alpha$ be an automorphism of $A$. Suppose $A$ has property (SI), and $\alpha$ has the approximate Rokhlin property. Then for every $m\in \N$ there exists $g_1,...,g_m, v$ in $F(A)$ such that
\begin{enumerate}
\item[(i)] $g_1,...,g_m$ are positive contractions;
\item[(ii)] $g_ig_j = 0$ whenever $1\leq i\neq j\leq m$;
\item[(iii)] $\alpha_\omega(g_k) = g_{k+1}\mod m\;\;\;\;\;\;\;\; (1\leq k\leq m)$;
\item[(iv)] $g_1v = v$;
\item[(v)] $\sum_{k=1}^m g_k + v^*v = 1$;
\item[(vi)] $\alpha_\omega^m(v) = v$.
\end{enumerate}
\end{prop}
\begin{proof}
Fix $m\in \N$ and $\e > 0$. By the $\e$-test (Lemma \ref{lem:epsilontest}) it suffices to construct $g_1,...,g_m, v$ in $F(A)$ satisfying (i)-(v) and
\begin{enumerate}
\item[(vi)'] $\| \alpha_\omega^m(v) - v \| < \e$.
\end{enumerate}
Choose $\ell\in \N$ sufficiently large so that $\frac{2}{\sqrt{\ell}} < \e$. Since $\alpha$ has the approximate Rokhlin property, we can find positive contractions $f_1,...,f_{\ell m}$ in $F(A)$ such that
\begin{enumerate}
\item[(1)] $\alpha_\omega(f_k) = f_{k+1}\mod \ell m\;\;\;\;\;\;\;\; (1\leq k\leq \ell m)$;
\item[(2)] $f_if_j = 0\;\;\;\;\;\;\;\; (1 \leq i\neq j\leq \ell m)$; 
\item[(3)] $e := 1 - \sum_{k=1}^{\ell m}f_k \in J_A$.
\end{enumerate}
Note that from (1) and (3) we have $\alpha_\omega(e) = e$. Let $\pi:F(A) \to \hm^\omega\cap \hm'$ be the canonical surjection. Observe that each $\pi(f_k)$ is a projection in $\hm^\omega\cap \hm'$ and
$$
\sum_{k=0}^{\ell m -1} \tilde{\alpha}_\omega^k( \pi(f_1))  = \sum_{k=1}^{\ell m}\pi(f_k) = 1.
$$
Since $A$ has property (SI), there exists $s\in F(A)$ such that $s^*s = e$ and $f_1s = s$. Define $g_1,...,g_m, v\in F(A)$ by
$$
g_k := \sum_{i=0}^{\ell-1} f_{k+im}\;\;\;\; \text{and}\;\;\;\; v := \frac{1}{\sqrt{\ell}}\sum_{i=0}^{\ell-1}\alpha_\omega^{im}(s)\;\;\;\;\;\; (1\leq k\leq m).
$$
We claim that these elements have the desired properties.

Conditions (i)-(iii) are obviously satisfied. Observe that whenever $1\leq i\neq j\leq \ell-1$,
$$
\alpha^{jm}_\omega(s^*)\alpha_\omega^{im}(s) = \alpha^{jm}_\omega(s^*)f_{jm}f_{im}\alpha_\omega^{im}(s) = 0.
$$
Therefore
$$
v^*v = \frac{1}{\ell}\sum_{i,j=0}^{\ell-1}\alpha_\omega^{jm}(s^*)\alpha_\omega^{im}(s) = \frac{1}{\ell}\sum_{i=0}^{\ell-1}\alpha_\omega^{im}(s^*s) = \frac{1}{\ell}\sum_{i=0}^{\ell-1}e = e.
$$
This observation directly implies condition (v). To verify condition (iv) one uses the facts that $f_1s = s$ and $f_if_j = 0$ whenever $i\neq j$. Finally (vi)' follows from a straightforward computation.
\end{proof}

The following lemma appeared in Kishimoto's work on one-sided shifts on UHF algebras (\cite{Ki02}), and also in \cite{Na01} and \cite{Ma01}. Roughly speaking, the result says that if an automorphism of a matrix algebra is ``almost'' a cyclic shift, then it has a certain kind of Rokhlin property.

\begin{lem}\label{lem:kishimoto} (cf. \cite[Lemma 2.2]{Ki02})
For any $m\in \N$ and $\e>0$, there exists a positive integer $\ell\in \N$ such that the following holds. If $\sigma = \Ad u$ is an automorphism of $M_{\ell m+1}$ implemented by a unitary $u$ with eigenvalues
$$
1, \omega^0, \omega^1,...,\omega^{\ell m-1},
$$
where $\omega = e^{2\pi i/\ell m}$ (here $i=\sqrt{-1}$), then there exist mutually orthogonal projections 
$$
p_1,p_2,...,p_m, q_1,q_2,...,q_{m+1}
$$
in $M_{\ell m + 1}$ such that
\begin{enumerate}
\item[(i)] $\| \sigma(p_k) - p_{k+1}\| < \e \mod m\;\;\;\;\;\;\;\; (1\leq k\leq m)$;
\item[(ii)] $\| \sigma(q_k) - q_{k+1}\| < \e \mod (m+1)\;\;\;\;\;\;\;\; (1\leq k\leq m+1)$;
\item[(iii)] $\sum_{k=1}^m p_k + \sum_{k=1}^{m+1}q_k = 1$.
\end{enumerate}
\end{lem}

In \cite{Ki01}, Kishimoto proved a Rokhlin type theorem by creating a matrix algebra inside the UHF algebra in a way that the action on the matrix algebra simulates what happens in Lemma \ref{lem:kishimoto}. The crucial observation of Matui and Sato is that, although in a general $C^*$-algebra there are no genuine matrix units, the simulation can still be done using order zero maps.

\begin{thm} (Matui-Sato)
\label{thm:kishimototrick}
Let $A$ be a unital simple separable $C^*$-algebra with $T(A)\neq \emptyset$. and let $\alpha$ be an automorphism of $A$. Suppose $A$ has property (SI), and $\alpha$ has the approximate Rokhlin property (Definition \ref{defn:approxRokhlin}). Then for any $m\in \N$ there exist positive contractions 
$$
a_1,...,a_m,b_1,...,b_m,c_1,...,c_{m+1}
$$
in $F(A)$ such that
\begin{enumerate}
\item[(i)] $\alpha_\omega(a_k) = a_{k+1},\;\;\; \alpha_\omega(b_k) = b_{k+1}\mod m\;\;\;\;\;\; (1\leq k\leq m)$;
\item[(ii)] $\alpha_\omega(c_k) = c_{k+1}\mod (m+1)\;\;\;\;\;\;(1\leq k\leq m+1)$;
\item[(iii)] $a_ia_j = 0,\;\;\; b_ib_j = 0\;\;\;\;\;\; (1\leq i\neq j\leq m)$;
\item[(iv)] $c_ic_j = 0\;\;\;\;\;\;\; (1\leq i\neq j\leq m+1)$;
\item[(v)] $b_ic_j = 0\;\;\;\;\;\; (1\leq i\leq m,\;\; 1\leq j\leq m+1)$;
\item[(vi)] $\sum_{k=1}^m a_k + \sum_{k=1}^m b_k + \sum_{k=1}^{m+1}c_k = 1$.
\end{enumerate}
In particular, $\rdim(\alpha) \leq 1$.
\end{thm}
\begin{proof}
Fix $m\in \N$ and $\e > 0$. Since we work in $F(A)$, it suffices to produce elements $a_k$, $b_k$ and $c_k$ satisfying (i)-(vi) up to $\e$. Let $\ell$ be the positive integer as in Lemma \ref{lem:kishimoto} that works for the $\e$ we fixed. By Proposition \ref{prop:mergebySI}, we can find elements $g_1,...,g_{\ell m}$ and $v$ in $F(A)$ such that
\begin{enumerate}
\item[(1)] $g_1,...,g_{\ell m}$ are positive contractions;
\item[(2)] $g_ig_j = 0$ whenever $1\leq i\neq j\leq \ell m$;
\item[(3)] $\alpha_\omega(g_k) = g_{k+1}\mod (\ell m)\;\;\;\;\;\;\;\; (1\leq k\leq \ell m)$;
\item[(4)] $g_1v = v$;
\item[(5)] $\sum_{k=1}^{\ell m} g_k + v^*v = 1$;
\item[(6)] $\alpha_\omega^{\ell m}(v) = v$.
\end{enumerate}
Define 
$$
x_1 := (v^*v)^{1/2},\;\;\;\;\;\; x_i = \alpha_\omega^{i-2}(v^*)\;\;\;\;\;\;\; (2\leq i\leq \ell m + 1).
$$
Then
$$
x_ix_i^* = \alpha_\omega^{i-2}(v^*v) = v^*v = x_1^*x_1\;\;\;\;\;\;\;\;\; (2\leq i\leq \ell m +1),
$$
and
\begin{align*}
(x_j^*x_j)(x_i^*x_i) = 0\;\;\;\;\;\;\;\; (1\leq i\neq j\leq \ell m + 1)
\end{align*}
(this follows from the facts that $(vv^*)(v^*v) = 0$, $g_1v = v$, and $g_ig_j = 0$ whenever $i\neq j$). Since $x_1\geq 0$ and $\|x_i\|\leq 1$ for each $1\leq i\leq \ell m+1$, by \cite[Proposition 2.4]{RoWi01} the map
$$
\psi:M_{\ell m+1}\to F(A),\;\;\;\; \psi(e_{ij}) = x_i^*x_j,
$$
where $(e_{ij})_{1\leq i,j\leq \ell m +1}$ are the standard matrix units, is a c.p.c order zero map.

Define an automorphism $\sigma\in \Aut(M_{\ell m+1})$ by $\sigma := \Ad u$, where
$$
u = \begin{pmatrix}
1 \\
 & 0 & 0 & \cdots & & 1 \\
 & 1 \\
 & & 1 \\
 & & & \ddots \\
 & & & & 1 & 0
\end{pmatrix}
$$
with respect to the standard matrix units. Notice that the eigenvalues of $u$ are $$
1, \omega^0, \omega^1,...,\omega^{\ell m-1},
$$
where $\omega = e^{2\pi \sqrt{-1}/\ell m}$. One checks that $\psi(\sigma(e_{ij})) = \alpha_\omega(\psi(e_{ij}))$ for all $i,j$, and hence $\psi\circ \sigma = \alpha_\omega\circ \psi$. By the choice of $\ell$, there exist projections
$$
p_1,...,p_m,q_1,...,q_{m+1}
$$
in $M_{\ell m + 1}$ such that
\begin{align*}
&\| \sigma(p_k) - p_{k+1} \| < \e \mod m\;\;\;\;\;\; (1\leq k\leq m);\\
&\| \sigma(q_k) - q_{k+1} \| < \e \mod (m+1)\;\;\;\;\;\; (1\leq k\leq m+1);\\
&\sum_{k=1}^m p_k + \sum_{k=1}^{m+1} q_k = 1.
\end{align*}
Let $h_j := \alpha^{j-1}(vv^*)$ for $1\leq j\leq \ell m$ and note that $h_j\leq g_j$ for each $j$. Define
\begin{align*}
b_k &:= \psi(p_k)\;\;\;\;\;\; (1\leq k\leq m),\\
c_k &:= \psi(q_k)\;\;\;\;\;\; (1\leq k\leq m+1),\\
a_k &:= \sum_{i=0}^{\ell-1} (g_{k+im} - h_{k+im})\;\;\;\;\;\; (1\leq k\leq m).
\end{align*}
We claim that these elements satisfy (i)-(vi) up to $\e$. It is clear that $\alpha_\omega(a_k) = a_{k+1}$ for $1\leq k\leq m$ and $a_ia_j = 0$ for $1\leq i\neq j\leq m$. The mutual orthogonality among $b_k$'s and $c_k$'s comes from the fact that $\psi$ is an order zero map. Also we have
\begin{align*}
\| \alpha_\omega(b_k) - b_{k+1} \| &= \| \alpha_\omega(\psi(p_k)) - \psi(p_{k+1}) \| \\
&= \| \psi(\sigma(p_k)) - \psi(p_{k+1}) \| \\
&\leq \|\sigma(p_k) - p_{k+1} \| \\
&< \e,
\end{align*}
and similarly $\| \alpha_\omega(c_k) - c_{k+1} \| < \e$. Finally,
\begin{align*}
\sum_{k=1}^m a_k &+ \sum_{k=1}^m b_k + \sum_{k=1}^{m+1} c_k \\
&= \sum_{k=1}^m\sum_{i=0}^{\ell -1} (g_{k+im}-h_{k+im}) + \sum_{k=1}^m \psi(p_k) + \sum_{k=1}^{m+1} \psi(g_k) \\
&= \sum_{k=1}^{\ell m} (g_k - h_k) + \psi(1) \\
&= \sum_{k=1}^{\ell m}g_k - \sum_{k=1}^{\ell m} \alpha^{k-1}(vv^*) + \sum_{k= 1}^{\ell m + 1} x_k^*x_k.
\end{align*}
Substituting the definitions of the $x_k$'s into the last expression and making cancellations, we see that the last expression is equal to $\sum_{k=1}^{\ell m}g_k + v^*v$, which is equal to 1 by construction.
\end{proof}

We have effectively proved Theorem \ref{thm:main}. The proof below summarizes the steps scattered in the paper.

\begin{proof} (of Theorem \ref{thm:main})
Write $(\hm, K, E)$ for the continuous $W^*$-bundle $(\overline{A}^u, \partial T(A), E)$ arising from the uniform 2-norm completion. Since $T(A) = T(A)^\alpha$, $\alpha$ induces an automorphism $(\tilde{\alpha}, \id)$ of $(\hm, K, E)$. Under the current assumptions each fiber $\pi_\lambda(\hm)$ is isomorphic to the hyperfinite II$_1$ factor $\mathcal{R}$, so by Proposition \ref{prop:fixpointalgebra} for each $p\in \N$ there is a unital embedding $\sigma: M_p\to (\hm^\omega\cap \hm')^{\tilde{\alpha}_\omega}$. Combining Theorem \ref{thm:bundleRokhlin} and Proposition \ref{prop:approxRokhlin} we see that $\alpha$ has the approximate Rokhlin property (Definition \ref{defn:approxRokhlin}). Since $A$ has property (SI), we can apply Theorem \ref{thm:kishimototrick}, which is exactly what we need.
\end{proof}

\begin{proof} (of Corollary \ref{cor:main}) It remains to show that under the assumption that $\partial T(A)$ is finite, the Rokhlin dimension of a strongly outer $\Z$-action is no greater than one. As in the proof of Theorem \ref{thm:main}, it suffices to establish the conclusion of Theorem \ref{thm:bundleRokhlin}. We shall follow the notations used in the statement and the proof of Theorem \ref{thm:bundleRokhlin}. In fact, it suffices to obtain the following slightly weaker conclusion: for every $p\in \N$, there exists a positive integer $n\geq p$ and projections $p_1,...,p_n$ in $\hm^\omega\cap \hm'$ such that
\begin{enumerate}
\item[(i)] $\beta_\omega(p_k) = p_{k+1}\mod n\;\;\;\; (1\leq k\leq n)$;
\item[(ii)] $\sum_{k=1}^n p_k = 1$;
\item[(iii)] $\| 1-p_1 \|_{2,u}^{(\omega)} < 1$.
\end{enumerate}
Once this is established, a similarly weaker version of Theorem \ref{thm:kishimototrick} would give us Rokhlin dimension no greater than one, by Remark \ref{rem:towerheight}.

Since $K$ is finite, the continuous $W^*$-bundle $(\hm,K,E) := (\overline{A}^u, \partial T(A), E)$ is nothing but a finite direct sum of the hyperfinite II$_1$ factor $\hr$. Therefore the conclusion above follows from the proof of \cite[Lemma 3.1]{Ki01} and \cite[Theorem 3.4]{MaSa02}. For reader's convenience we include a proof here.

As in the proof of Theorem \ref{thm:bundleRokhlin}, we will find positive contractions $r_1,...,r_n$ in $\hm$ satisfying conditions (a)-(e). Let $s = \mathrm{card}(K)$. First assume that the homeomorphism $\sigma:K\to K$ acts as a cycle (in the sense of permutation theory). Choose a positive integer $n\geq p$ such that $n \equiv 1\mod s$ and fix $\lambda\in K$. Since each power of the automorphism $\beta_\lambda^s \in \Aut(\pi_\lambda(\hm))$ is outer, we can apply Connes' non-commutative Rokhlin theorem \cite{Con01} and obtain mutually orthogonal projections $q_1,...,q_n$ in $\pi_\lambda(\hm)$ such that
\begin{enumerate}
\item[(a)'] $\| \beta_\lambda^s(q_k) - q_{k+1}\|_{2,\lambda} < \e'\mod n\;\;\;\;\;\;(1\leq k\leq n)$;
\item[(c)'] $\| [q_k, \pi_\lambda(x) ] \|_{2,\lambda} < \e'\;\;\;\;\;\; (1\leq k\leq n, x\in \Omega)$
\item[(d)'] $\sum_{k=1}^nq_k = 1$
\end{enumerate}
for any prescribed $\e' > 0$. Since $K$ is finite, for each $k$ there is a positive contraction $r_k'$ in $\hm$ such that $\pi_\lambda(r_k') = q_k$ and $\pi_\mu(r_k') = 0$ whenever $\mu\neq \lambda$. Define
$$
r_1 := \sum_{j=0}^{s-1}\beta^{nj}(r_1')\;\;\;\mathrm{and}\;\;\; r_k = \beta^{k-1}(r_1)\;\;\;\;\;\; (2\leq k\leq n).
$$
One checks that $r_1,...,r_n$ are positive contractions in $\hm$ satisfying conditions (a)-(d). Condition (e) follows from a calculation/approximation: for every $\mu\in K$ we have
\begin{align*}
n\tau_\mu(r_1) &= n\sum_{j=0}^{s-1}\tau_\mu(\beta^{nj}(r_1')) \\
&= n\sum_{j=0}^{s-1}\tau_\mu(\beta^j(r_1'))\;\;\;\;\;\; (n \equiv 1 \mod s) \\
&= \sum_{j=0}^{ns-1}\tau_\mu(\beta^j(r_1'))\;\;\;\;\;\; (n\equiv 1\mod s) \\
&= \tau_\mu\left( \sum_{k=1}^n r_k \right) \approx 1_\hm\;\;\;\;\;\; (\mathrm{condition (d)}).
\end{align*}
In the general case, the homeomorphism $\sigma$ can be decomposed into disjoint cycles. Let $\ell$ be a positive integer such that $\sigma^\ell = \id$. Now the proof is finished by choosing $n\geq p$ such that $n\equiv 1$ (mod $\ell$) and repeating the same construction as in the previous paragraph with $s$ replaced by $\ell$.
\end{proof}

\begin{rem}
In the above proof, using a more involved argument one can actually obtain the conclusion of Theorem \ref{thm:bundleRokhlin} for any $p\in \N$, not just some sufficiently large $n\geq p$. For details, we refer the reader to the proof of \cite[Theorem 3.6]{MaSa03}.
\end{rem}


\end{document}